\documentclass[12pt,a4paper]{article} 
\usepackage{amsmath}
\usepackage{amsthm}
\usepackage{amsfonts}
\usepackage{amssymb}
\usepackage[colorlinks]{hyperref}
\usepackage{stmaryrd}
\usepackage{bussproofs}
\usepackage{euscript} 
\usepackage{color}
\usepackage{float}

\theoremstyle{plain} 
\newtheorem{theorem}{Theorem}[section]

\newtheorem{lemma}[theorem]{Lemma}

\newtheorem{corollary}[theorem]{Corollary}
\theoremstyle{definition}
\newtheorem{definition}[theorem]{Definition}

\newtheorem{remark}[theorem]{Remark}

\newcommand{\imp}{\rightarrow} 
\newcommand{\Ifff}{\Leftrightarrow}
\newcommand{\en}{\wedge} 
 
\newcommand{\ifff}{\leftrightarrow}
\newcommand{\E}{\exists}
\newcommand{\A}{\forall} 
\newcommand{\Ap}{\forall^{\circ}\hspace{-.3mm}p\hspace{.2mm}}
\newcommand{\Apd}{\forall^{\diamond}\hspace{-.3mm}p\hspace{.2mm}}
\newcommand{\Apm}{\forall^{\circ}\!\!\!_{{\cal M}} \hspace{-.08mm}p\hspace{.2mm}} 
\newcommand{\Apmcf}{\forall\!_{{\cal M}} \hspace{-.08mm}p\hspace{.2mm}} 
\newcommand{\Apax}{\forall^{\circ}\!\!\!_{{\cal A}} \hspace{-.08mm}p\hspace{.2mm}}
\newcommand{\Ep}{\exists^{\circ}\hspace{-.3mm}p\hspace{.2mm}}
\newcommand{\Apr}{\forall^{\circ}\!\!\!_{{\cal R}} \hspace{-.08mm}p\hspace{.2mm}}

\newcommand{\seq}{\Rightarrow}

\newcommand{\cald}{{\EuScript D}}

\newcommand{\De}{\Delta}
\newcommand{\Ga}{\Gamma}

\renewcommand{\phi}{\varphi}

\newcommand{\Gthc}{\mathbf{G3cp}}
\newcommand{\GthW}{\mathbf{G3W}}
\newcommand{\GCE}{\mathbf{GCE}}
\newcommand{\GCM}{\mathbf{GCM}}
\newcommand{\GCMC}{\mathbf{GCMC}}

\newcommand{\GCEN}{\mathbf{GCEN}}
\newcommand{\GCMN}{\mathbf{GCMN}}
\newcommand{\GCMCN}{\mathbf{GCMCN}}

\begin{document}
\title{Uniform Lyndon Interpolation for Basic Non-normal Modal and Conditional Logics\footnote{Support by the Netherlands Organisation for Scientific Research under grant 639.073.807 is gratefully acknowledged.}}

\author{Amirhossein Akbar Tabatabai,
Rosalie Iemhoff,
Raheleh Jalali\\
Utrecht University}

\maketitle 

\begin{abstract}
In this paper, a proof-theoretic method to prove uniform Lyndon interpolation for non-normal modal and conditional logics is introduced and applied to show that the logics $\mathsf{E}$, $\mathsf{M}$, $\mathsf{EN}$, $\mathsf{MN}$, $\mathsf{MC}$, $\mathsf{K}$, and their conditional versions, $\mathsf{CE}$, $\mathsf{CM}$, $\mathsf{CEN}$, $\mathsf{CMN}$, $\mathsf{CMC}$, $\mathsf{CK}$, in addition to $\mathsf{CKID}$ have that property. In particular, it implies that these logics have uniform interpolation. Although for some of them the latter is known, the fact that they have uniform Lyndon interpolation is new. Also, the proof-theoretic proofs of these facts are new, as well as the constructive way to explicitly compute the interpolants that they provide. On the negative side, it is shown that the logics $\mathsf{CKCEM}$ and $\mathsf{CKCEMID}$ enjoy uniform interpolation but not uniform Lyndon interpolation. Moreover, it is proved that the non-normal modal logics $\mathsf{EC}$ and $\mathsf{ECN}$ and their conditional versions, $\mathsf{CEC}$ and $\mathsf{CECN}$, do not have Craig interpolation, and whence no uniform (Lyndon) interpolation. 
\\

\noindent \textbf{keywords.} non-normal modal logics, conditional logics, uniform interpolation, uniform Lyndon interpolation, Craig interpolation.
\end{abstract}


\section{Introduction}
Uniform interpolation (UIP), a strengthening of interpolation in which the interpolant only depends on the premise or the conclusion of an implication, is an intriguing logical property. One of the reasons 
is that it is hard to predict which logic does have the property and which does not. Well-behaved logics like {\sf K} and {\sf KD} have it, but then, other well-known modal logics, such as {\sf K4}, do not.  Early results on the subject were by Shavrukov \cite{Shavrukov}, who proved 
UIP for the modal logic  {\sf GL}, and by Ghilardi \cite{Ghilardi} and Visser \cite{Visser}, who independently proved the same for {\sf K}, followed later by B\'ilkov\'a, who showed that {\sf KT} has the property, as well \cite{Bilkova}. Surprisingly, {\sf K4} and {\sf S4} do not have UIP, 
although they do have interpolation \cite{Bilkova,Ghilardi}. Pitts provided the first proof-theoretic proof of UIP, 
for intuitionistic propositional logic, {\sf IPC}, the smallest intermediate logic \cite{Pitts}. Results from \cite{Zawadowski,Maksimova} imply that there are exactly seven intermediate logics with interpolation and that they are exactly the intermediate logics with UIP. 
Pitts' result is especially important to us, as also in our paper the approach is proof-theoretic. \\

In this paper we investigate UIP in the setting of non-normal modal and conditional logics. For these classes of logics, the study of UIP has a more recent history, and the property is less explored than in the setting of normal logics. Before we discuss our own contributions, we briefly describe the logics and the results from the literature that are relevant to our investigation.  \\

Non-normal modal logics are modal logics in which the $K$-axiom, i.e.\ the axiom $\Box(\phi \imp \psi) \imp (\Box\phi \imp \Box\psi)$, does not hold, but a weaker version that is given by the following $E$-rule does:
\begin{center}
 \AxiomC{$\phi \ifff \psi$}
 \UnaryInfC{$\Box\phi \ifff \Box\psi$}
 \DisplayProof
\end{center}
Thus the minimal non-normal modal logic, {\sf E}, is propositional logic plus the $E$-rule above. 
Over the last decades non-normal modal logics have emerged in various fields, such as game theory and epistemic and deontic logic \cite{Chellas}. Two well-known non-normal modal logics that are investigated in this paper are natural weakenings of the principle 
$\Box(\phi \en\psi) \ifff (\Box\phi \en\Box\psi)$ that implies $K$ over {\sf E}. Namely, the two principles: 
\[
 (M) \ \ \ \Box(\phi \en\psi) \imp (\Box\phi \en\Box\psi) 
 \ \ \ \ \ 
 (C) \ \ \ (\Box\phi \en\Box\psi) \imp \Box(\phi \en\psi).  
\]
Because the $K$-axiom holds in the traditional relational semantics for modal logic, non-normal modal logics require different semantics, of which the most well-known is neighborhood semantics. As we do not need semantics in this paper, we refer the interested reader to the textbook \cite{Pacuit}.\\

Conditional logics formalize reasoning about conditional statements, which are statements of the form {\em if $\phi$ were the case, then $\psi$}, denoted as $\phi \triangleright \psi$. This includes formal reasoning about counterfactuals, which are instances of conditional statements in which the antecedent is false. Examples of well-known conditional logics are the logics {\sf CE} and {\sf CK}. Logic {\sf CE} consists of classical propositional logic plus the following {\it CE}-rule, that could be seen as the conditional counterpart of the $E$-rule above (by reading the $\phi_i$ as $\top$):
\begin{center}
 \AxiomC{$\phi_0 \ifff \phi_1$}
 \AxiomC{$\psi_0 \ifff \psi_1$}
 \BinaryInfC{$(\phi_1 \triangleright \psi_1 ) \ifff (\phi_0 \triangleright \psi_0 )$}
 \DisplayProof
\end{center}
Conditional logic {\sf CK} is the result of adding to classical propositional logic the following {\it CK}-rule, which derives the conditional counterpart of $K$: 
\begin{center}
 \AxiomC{$\{\phi_0 \ifff \phi_i\}_{i \in I}$}
 \AxiomC{$\{\psi_i\}_{i \in I} \imp \psi_0$}
 \BinaryInfC{$\{\phi_i \triangleright \psi_i\}_{i \in I} \imp (\phi_0 \triangleright \psi_0 )$}
 \DisplayProof
\end{center}
where $I$ is a possibly empty set. Conditional logics have played an important role in philosophy ever since their introduction by Lewis in 1973 \cite{Lewis}. Another area where they appear is artificial intelligence, in the setting of nonmonotonic reasoning \cite{FriedmanHalpern}. 
The first formal semantics and axiomatization of conditionals appeared in \cite{Stalnaker} and in the years after that the logics were further investigated in \cite{Chellas, Lewis, NuteCross}.\\ 

For several well-known non-normal logics UIP 
has been established, for example, for the monotone logic {\sf M} \cite{Venema}, a result later extended in \cite{Pattinson,Seifan} to other non-normal modal and conditional logics, such as {\sf E} and {\sf CK}. 
Recently the proof theory of conditionals has been explored in \cite{LellmannPattinson,Olivetti,PattinsonSchroeder,SchroederPattinsonHausmann} where sequent calculi and an automated theorem prover for several conditional logics have been developed. 
Particularly relevant for this paper are the results in \cite{Pattinson,PattinsonSchroeder}, where in the former it is shown that the conditional logic {\sf CK} has UIP, and in the latter, sequent-style systems for various conditional logics are provided. \\
Our interest in the property of UIP for non-normal modal and conditional logics lies in the fact that it can be used as a tool in what we would like to call {\em universal proof theory}, the area where one is concerned with the general behavior of proof systems, investigating problems such as the existence problem (when does a theory have a certain type of proof system?) and the equivalence problem (when are two proof systems equivalent?). The value of UIP 
for the existence problem has been addressed in a series of recent papers in which a method is 
developed to prove UIP 
that applies to many intermediate, (intuitionistic) modal, and substructural (modal) logics \cite{Tab,Iemhoff,Iemhoffa}. The proof-theoretic method makes use of 
sequent calculi, and shows that general conditions on the calculi imply UIP 
for the corresponding logic. Thus implying that any logic without UIP 
cannot have a sequent calculus 
satisfying these conditions. The generality of the conditions, such as closure under weakening, makes this into a powerful tool, especially for those classes of logics in which UIP 
is rare, such as intermediate logics. 
Note that in principle other regular properties than UIP 
could be used in this method, as long as the property is sufficiently rare to be of use. \\

In this paper, we do not focus on the connection with the existence problem as just described, but rather aim to show the flexibility and utility 
of our method to prove UIP 
by showing that it can be extended to two other classes  of logics, namely the class of non-normal modal logics and the class of conditional logics. And by furthermore showing that it is constructive and 
can be easily adapted to prove not only UIP 
but even uniform Lyndon interpolation. Uniform Lyndon interpolation (ULIP) is a strengthening of UIP  
in which the interpolant respects the polarity of propositional variables (a definition follows in the next section). It first occurred in \cite{Kurahashi}, where it was shown that several normal modal logics, including $\mathsf{K}$ and $\mathsf{KD}$, have that property. In this paper we show that the non-normal modal logics $\mathsf{E}$, $\mathsf{M}$, $\mathsf{EN}$, $\mathsf{MN}$, $\mathsf{MC}$, $\mathsf{K}$, their conditional versions, $\mathsf{CE}$, $\mathsf{CM}$, $\mathsf{CEN}$, $\mathsf{CMN}$, $\mathsf{CMC}$, $\mathsf{CK}$ in addition to $\mathsf{CKID}$ have uniform Lyndon interpolation and 
the interpolant can be constructed explicitly from the proof. In the last part of this paper we show that the non-normal modal logics $\mathsf{EC}$ and $\mathsf{ECN}$ and the conditional logics $\mathsf{CEC}$ and $\mathsf{CECN}$ do not have Craig interpolation, and whence no uniform (Lyndon) interpolation either. This surprising fact makes $\mathsf{EC}$, $\mathsf{ECN}$, $\mathsf{CEC}$ and $\mathsf{CECN}$ potential candidates for our approach to the existence problem 
discussed above, but that we have to leave for another paper. Interestingly, the conditional logics {\sf CKCEM} and {\sf CKCEMID} have UIP but not ULIP, as we will show at the end of the paper. Some of these results have already been obtained in the literature. 
That the logics $\mathsf{E}$, $\mathsf{M}$, and $\mathsf{MC}$ have UIP has already been established in \cite{Pattinson,Venema}, but that they have uniform Lyndon interpolation is, as far as we know, a new insight. In \cite{Orlandelli} it is shown that $\mathsf{E}$, $\mathsf{M}$, $\mathsf{MC}$, $\mathsf{EN}$, $\mathsf{MN}$ have Craig interpolation. The proof that these logics have UIP is not a mere extension of the proof that they have interpolation but requires a different approach.\\

Our proof-theoretic method to prove UIP makes use of sequent calculi for non-normal modal and conditional logics that are equivalent or equal to calculi introduced in \cite{Orlandelli}. Only for the five logics {\sf CE}, {\sf CM}, {\sf CEN}, {\sf CMN}, and {\sf CMC}, we had to develop cut-free sequent calculi ourselves. 
Interestingly, for logics with Lyndon interpolation (LIP),  
that fact does not always follow easily from the proof that they have Craig interpolation (CIP),  
as is for example the case for {\sf GL} \cite{Shamkanov}. But for our method this indeed is the case: a proof of UIP using this method easily implies ULIP. Thus, the hard work lies in proving the former. Until now such proofs have always been semantical in nature. 
For these reasons, we consider the proof-theoretic method to prove uniform interpolation for non-normal modal and conditional logics the main contribution of this paper. In \cite{Seifan} the search for proof-theoretic techniques to prove uniform interpolation in the setting of non-normal modal logics is explicitly mentioned in the conclusion of that paper. \\

The paper is build up as follows. In Section \ref{Preliminaries}, the logics and their sequent calculi are defined, except for the five logics for which we develop sequent calculi ourselves, which are treated in Section \ref{SequentSystems}. This section also contains the proofs of cut-elimination for the new five calculi. Section \ref{ULIPSection} contains the positive results showing that most logics that we consider have ULIP. It also proves that the logics $\mathsf{CKCEM}$ and $\mathsf{CKCEMID}$ have UIP (but not ULIP). Section \ref{SecNegativeResults} is dedicated to the negative results, showing that the non-normal modal logics $\mathsf{EC}$ and $\mathsf{ECN}$ and the corresponding conditional logics $\mathsf{CEC}$ and $\mathsf{CECN}$ do not have CIP. It also shows that if an extension of $\mathsf{CKCEM}$ enjoys ULIP, then its conditional behaves in a peculiar manner. As a consequence, we conclude with the proof that the logics $\mathsf{CKCEM}$ and $\mathsf{CKCEMID}$ do not have ULIP.

\subsection*{Acknowledgements}
We thank Iris van der Giessen for fruitful discussions on the topic of this paper and three referees for comments that helped improving the paper.

\section{Preliminaries}\label{Preliminaries}
Set $\mathcal{L}_{\Box}=\{\wedge, \vee, \to, \bot, \Box\}$ as the language of modal logics and $\mathcal{L}_{\triangleright}=\{\wedge, \vee, \to, \bot, \triangleright\}$ as the language of conditional logics. We write $\mathcal{L}$ to refer to both of these languages, depending on the context. We use $\top$ and $\neg A$ as abbreviations for $\bot \to \bot$ and $A \to \bot$, respectively, and write $\phi \in \mathcal{L}$ to indicate that $\phi$ is a formula in the language $\mathcal{L}$. The \emph{weight of a formula} is defined inductively as follows. Define $w(\bot)=w(p)=0$, for any atomic formula $p$, $w(A \odot B)=w(A)+w(B)+1$, for any $\odot \in \{\wedge, \vee, \to\}$, and depending on the language $\mathcal{L}$, $w(\Box A)=w(A)+1$ and $w(A \triangleright B)=w(A)+w(B)+1$, for the non-propositional connectives.
\begin{definition}
The sets of positive and negative variables of a formula $\phi \in \mathcal{L}$, denoted respectively by $V^+(\phi)$ and $V^-(\phi)$ are defined recursively by:

\begin{itemize}
\item[$\bullet$]
$V^+(p)=\{p\}$, $V^-(p)=
V^+(\top)=V^-(\top)=V^+(\bot)=V^-(\bot)=\varnothing$, for any atomic  
formula $p$,
\item[$\bullet$]

$V^+(\phi \odot \psi)=V^+(\phi) \cup V^+(\psi)$ and  $V^-(\phi \odot \psi)=V^-(\phi) \cup V^-(\psi)$, for 
$\odot \in \{\wedge, \vee\}$, 
\item[$\bullet$]

$V^+(\phi \to \psi)=V^-(\phi) \cup V^+(\psi)$ and  $V^-(\phi \to \psi)=V^+(\phi) \cup V^-(\psi)$,
\item[$\bullet$]

$V^+(\Box \phi)=V^+(\phi)$  and $V^-(\Box \phi)=V^-(\phi)$, for $\mathcal{L}=\mathcal{L}_{\Box}$.

\item[$\bullet$]

$V^+(\phi \triangleright \psi)=V^-(\phi) \cup V^+(\psi)$ and  $V^-(\phi \triangleright \psi)=V^+(\phi) \cup V^-(\psi)$, for $\mathcal{L}=\mathcal{L}_{\triangleright}$.
\end{itemize}
\noindent Define $V(\phi)$ as $V^+(\phi) \cup V^-(\phi)$. For an atomic formula $p$, a formula $\phi$ is called {\em $p^+$-free} (resp., {\em $p^-$-free}), if $p \notin V^+(\phi)$ (resp., $p \notin V^-(\phi)$). It is called {\em $p$-free} if $p \notin V(\phi)$. Note that a formula is $p$-free iff $p$ does not occur in it.
\end{definition}

For the sake of brevity, when we want to refer to both $V^+(\phi)$ and $V^-(\phi)$, we use the notation $V^{\dagger}(\phi)$ with the condition ``for any $\dagger \in \{+, -\}$". If we 
want to refer to one of $V^+(\phi)$ and $V^-(\phi)$ and its dual, we write $V^{\circ}(\phi)$ for the one we intend 
and $V^{\diamond}(\phi)$\footnote{The label $\diamond$ has nothing to do with the modal operator $\Diamond =\neg\Box\neg$.} for the other one. 
For instance, if we state that for any atomic formula $p$, any $\circ \in \{+, -\}$ and any $p^{\circ}$-free formula $\phi$, there is a $p^{\diamond}$-free formula $\psi$ such that $\phi \vee \psi \in L$, we are actually stating that if $\phi$ is $p^+$-free, there is a $p^-$-free $\psi$ such that $\phi \vee \psi \in L$ and if $\phi$ is $p^-$-free, there is a $p^+$-free $\psi$ such that $\phi \vee \psi \in L$.

\begin{definition} \label{Logic}
A \emph{logic} $L$ is a set of formulas in $\mathcal{L}$ extending the set of classical tautologies, $\mathsf{CPC}$, and closed under substitution and modus ponens $\phi, \phi \to \psi \vdash \psi$.
\end{definition}

\begin{definition} \label{DfnLyndonInterpolation}
A logic $L$ has \emph{Lyndon interpolation property (LIP)} if for any formulas $\phi, \psi \in \mathcal{L}$ such that $L \vdash \phi \to \psi$, there is a formula $\theta \in \mathcal{L}$ such that $V^{\dagger}(\theta) \subseteq V^{\dagger}(\phi) \cap V^{\dagger}(\psi)$, for any $\dagger \in \{+, -\}$ and $L \vdash \phi \to \theta$ and $L \vdash \theta \to \psi$. A logic has \emph{Craig interpolation (CIP)} if it has 
the above properties, omitting all the superscripts $\dagger \in \{+, -\}$. 
\end{definition}

\begin{definition} \label{DfnUniformInterpolation} 
A logic $L$ has \emph{uniform Lyndon interpolation property (ULIP)} if for any formula $\phi \in \mathcal{L}$, 
atom 
$p$, and 
$\circ \in \{+, -\}$, there are 
$p^{\circ}$-free formulas, 
$\Ap \phi$ and $\Ep \phi$, such that $V^{\dagger}(\Ep \phi) \subseteq V^{\dagger}(\phi)$ and $V^{\dagger}(\Ap \phi) \subseteq V^{\dagger}(\phi)$, for any 
$\dagger \in \{+, -\}$ and
\begin{description}
\item[$(i)$]
$L \vdash \Ap \phi \to \phi$,
\item[$(ii)$]
for any $p^{\circ}$-free formula $\psi$ if $L \vdash \psi \to \phi$ then $L \vdash \psi \to \Ap \phi$,
\item[$(iii)$]
$L \vdash \phi \to \Ep \phi$, and
\item[$(iv)$]
for any $p^{\circ}$-free formula $\psi$ if $L \vdash \phi \to \psi$ then $L \vdash \Ep \phi \to \psi$.
\end{description}
A logic has \emph{uniform interpolation property (UIP)} if it has all the above properties, omitting the superscripts $\circ, \dagger \in \{+, -\}$, everywhere.
\end{definition}

\begin{remark}
As the formulas $\Ap \phi$ and $\Ep \phi$ are provably unique, using the functional notation of writing $\Ap \phi$ and $\Ep \phi$ as the functions with the arguments ${\circ} \in \{+, -\}$, $p$ and $\phi$ is allowed.
\end{remark}

\begin{theorem}
If a logic $L$ has ULIP, 
then it has both LIP and UIP.
\end{theorem}
\begin{proof}
For UIP, 
set $\forall p \phi=\forall^+ p \forall^- p \phi$ and $\exists p \phi=\exists^+ p \exists^- p \phi$. We 
only prove the claim for $\forall p \phi$, as the case for $\exists p \phi$ is similar. 
First, it is clear that $V^{\dagger}(\forall p \phi) \subseteq V^{\dagger}(\phi)$, for any $\dagger \in \{+, -\}$. Hence, we have $V(\forall p \phi) \subseteq V(\phi)$. Moreover, $\forall p \phi$ is $p$-free. Because $\forall^- p \phi$ is $p^-$-free by definition and as $V^-(\forall p \phi) \subseteq V^-(\forall^- p\phi)$, the formula $\forall^+ p \forall^- p \phi$ is also $p^-$-free. As $\forall^+ p \forall^- p \phi$ is $p^+$-free by definition, we have $p \notin V(\forall p \phi)=V^+(\forall p \phi) \cup V^-(\forall p \phi)$. For 
condition $(i)$ in Definition \ref{DfnUniformInterpolation}, as $L \vdash \forall^+ p \forall^-p \phi \to \forall^- p \phi$ and $L \vdash \forall^- p \phi \to \phi $, we have $L \vdash \forall p \phi \to \phi$. For 
condition $(ii)$, if $L \vdash \psi \to \phi$, for a $p$-free $\psi$, then $\psi$ is also $p^-$-free and hence $L \vdash \psi \to \forall^- p \phi$. As $\psi$ is also $p^+$-free, we have $L \vdash \psi \to \forall^+ p \forall^- p \phi$.\\
For LIP, 
assume $L \vdash \phi \to \psi$. 
For any $\dagger \in \{+, -\}$, set $P^{\dagger}=V^{\dagger}(\phi)-[V^{\dagger}(\phi) \cap V^{\dagger}(\psi)]$. Define $\theta=\exists^+ P^+ \exists^- P^- \phi$, where by $\exists^{\dagger} \{p_1, \dots, p_n\}^{\dagger}$ we mean $\E p^{\dagger}_1 \dots \E p^{\dagger}_n$. 
Since $\theta$ is $p^{\dagger}$-free for any $p \in P^{\dagger}$ and any $\dagger \in \{+, -\}$, we have $V^{\dagger}(\theta) \subseteq V^{\dagger}(\phi)-P^{\dagger} \subseteq V^{\dagger}(\phi) \cap V^{\dagger}(\psi)$. 
For the provability condition, it is clear that $L \vdash \phi \to \theta$ and as $\psi$ is $p^{\dagger}$-free for any $p \in P^{\dagger}$, we have $L \vdash \theta \to \psi$.
\end{proof}

\subsection{Non-normal modal and conditional logics}
In this subsection, we recall the basic non-normal modal and conditional logics via their Hilbert-style presentations. Let us start with non-normal modal logics. The logic $\mathsf{E}$ is defined as the smallest set of formulas in $\mathcal{L}_{\Box}$ containing all classical tautologies and closed under the following two rules:
\begin{center}
    \begin{tabular}{c c}
        \AxiomC{$\phi$} 
 \AxiomC{$\phi \to \psi$} 
   \RightLabel{\footnotesize$MP$} 
 \BinaryInfC{$\psi$}
 \DisplayProof
 & \quad
 \AxiomC{$\phi \leftrightarrow \psi$}
 \RightLabel{\footnotesize$E$} 
 \UnaryInfC{$\Box \phi \leftrightarrow \Box \psi$}
 \DisplayProof
    \end{tabular}
\end{center}
We take $\mathsf{E}$ as the base logic and define the other non-normal modal logics as depicted in Figure \ref{figModalHilbert}, by adding the following modal axioms:
\begin{center}
\begin{tabular}{c c c}
$\Box (\phi \wedge \psi) \to \Box \phi \wedge \Box \psi$ \, \footnotesize(M)
 &
 $\Box \phi \wedge \Box \psi \to \Box (\phi \wedge \psi) $ \, \footnotesize(C)
  &
 $\Box \top$ \, \footnotesize(N)
\end{tabular}
\end{center}
    \begin{figure}[H]
\begin{center}
    \begin{tabular}{c c}
       $\mathsf{EN}= \mathsf{E}+ (N)$
 & \quad
 $\mathsf{M}= \mathsf{E}+ (M)$ 
    \end{tabular}
\end{center}
\begin{center}
    \begin{tabular}{c c}
       $\mathsf{MN}= \mathsf{M}+ (N)$
 & \quad
 $\mathsf{MC}= \mathsf{M}+ (C)$ 
    \end{tabular}
\end{center}
\begin{center}
    \begin{tabular}{c c}
       $\mathsf{K}= \mathsf{MC}+ (N)$
 & \quad
 $\mathsf{EC}= \mathsf{E}+ (C)$ 
    \end{tabular}
\end{center}
\begin{center}
    \begin{tabular}{c c}
       $\mathsf{ECN}= \mathsf{EC}+ (N)$ 
    \end{tabular}
\end{center}
\caption{Non-normal modal logics}
    \label{figModalHilbert}
\end{figure}
Similarly, for conditional logics, define $\mathsf{CE}$ as the smallest set of formulas in $\mathcal{L}_{\triangleright}$ containing all classical tautologies and closed under the following two rules: 
\begin{center}
    \begin{tabular}{c c}
        \AxiomC{$\phi$} 
 \AxiomC{$\phi \to \psi$} 
   \RightLabel{\footnotesize$MP$} 
 \BinaryInfC{$\psi$}
 \DisplayProof
 & \quad
 \AxiomC{$\phi_0 \leftrightarrow \phi_1$} 
 \AxiomC{$\psi_0 \leftrightarrow \psi_1$} 
   \RightLabel{\footnotesize$CE$} 
 \BinaryInfC{$\phi_0 \triangleright \psi_0 \to \phi_1 \triangleright \psi_1$}
 \DisplayProof
    \end{tabular}
\end{center}
The other conditional logics are defined as depicted in Figure \ref{figConditionalLogics}, by adding the following conditional axioms to $\mathsf{CE}$:\\

\begin{center}
$(\phi \triangleright \psi \wedge \theta) \to (\phi \triangleright \psi) \wedge (\phi \triangleright \theta)$ \quad \footnotesize(CM)
\end{center}
\begin{center}
    $(\phi \triangleright \psi) \wedge (\phi \triangleright \theta) \to (\phi \triangleright \psi \wedge \theta) $ \quad \footnotesize(CC)
\end{center}
\begin{center}
\begin{tabular}{c c c}
 $\phi \triangleright \top$\;  \footnotesize(CN)
 &
  $(\phi \triangleright \psi) \vee (\phi \triangleright \neg \psi)$ \; \footnotesize(CEM)
  &
 $\phi \triangleright \phi$ \; \footnotesize(ID)
\end{tabular}
\end{center}
\begin{figure}[H]
\begin{center}
    \begin{tabular}{c c}
       $\mathsf{CEN}= \mathsf{CE}+ (CN)$ 
 & \quad
 $\mathsf{CM}= \mathsf{CE}+ (CM)$ 
    \end{tabular}
\end{center}
\begin{center}
    \begin{tabular}{c c}
       $\mathsf{CMN}= \mathsf{CM}+ (CN)$ 
 & \quad
 $\mathsf{CMC}= \mathsf{CM}+ (CC)$ 
    \end{tabular}
\end{center}
\begin{center}
    \begin{tabular}{c c}
       $\mathsf{CK}= \mathsf{CMC}+ (CN)$ 
 & \quad
 $\mathsf{CEC}= \mathsf{CE}+ (CC)$ 
    \end{tabular}
\end{center}
\begin{center}
    \begin{tabular}{c c}
       $\mathsf{CECN}= \mathsf{CEC}+ (CN)$ 
 & \quad
 $\mathsf{CKID}= \mathsf{CK}+ (ID)$ 
    \end{tabular}
\end{center}
\begin{center}
    \begin{tabular}{c c}
       $\mathsf{CKCEM}= \mathsf{CK}+ (CEM)$ 
 & \quad
 $\mathsf{CKCEMID}= \mathsf{CKCEM}+ (ID)$ 
    \end{tabular}
\end{center}
\caption{Some conditional logics}
\label{figConditionalLogics}
\end{figure}

\subsection{Sequent calculi}
We use capital Greek letters and the bar notation in $\bar{\phi}$ and $\bar{C}$ to denote multisets. 
A \emph{sequent} 
is an expression in the form $\Gamma \Rightarrow \Delta$, where $\Gamma$ (the \emph{antecedent}) and $\Delta$ (the \emph{succedent}) are multisets of formulas. It is interpreted as $\bigwedge\Gamma \imp \bigvee \Delta$. 
For sequents $S=(\Gamma \Rightarrow \Delta)$ and $T=(\Pi \Rightarrow \Lambda)$ 
we denote the sequent $\Gamma, \Pi \Rightarrow \Delta, \Lambda$ by $S \cdot T$, and the multisets $\Gamma$ and $\Delta$ by $S^a$ and $S^s$, respectively. Define $V^+(S)=V^-(S^a) \cup V^+(S^s)$ and $V^-(S)=V^+(S^a) \cup V^-(S^s)$ and the \emph{weight of a sequent} as the sum of the weights of the formulas occurring in that sequent. 
A sequent $S$ is \emph{lower than} a sequent $T$, if the weight of $S$ is less than the weight of $T$.\\

\noindent In this paper we are interested in modal and conditional extensions of the well-known sequent calculus $\Gthc$ from \cite{Troelstra} 
(Figure~\ref{figgthc}) for classical logic $\mathsf{CPC}$:
\begin{figure}[H]
\begin{center}
 \begin{tabular}{cc}
 \AxiomC{}
  \RightLabel{\footnotesize$ Ax$}
 \UnaryInfC{$\Gamma, p \Rightarrow p, \Delta$}
 \DisplayProof
 &
 \AxiomC{}
 \RightLabel{\footnotesize$L \bot$}
 \UnaryInfC{$\Gamma, \bot \Rightarrow, \Delta$}
 \DisplayProof
 \\[3ex]
 \AxiomC{$\Gamma, \phi, \psi \Rightarrow \Delta$}
 \RightLabel{\footnotesize$L \wedge$} 
 \UnaryInfC{$\Gamma, \phi \wedge \psi \Rightarrow \Delta$}
 \DisplayProof
 &
 \AxiomC{$\Gamma \Rightarrow \phi, \Delta$}
 \AxiomC{$\Gamma \Rightarrow \psi, \Delta$}
 \RightLabel{\footnotesize$R \wedge$} 
 \BinaryInfC{$\Gamma \Rightarrow \phi \wedge \psi, \Delta$}
 \DisplayProof
  \\[3ex]
 \AxiomC{$\Gamma, \phi \Rightarrow \Delta$}
 \AxiomC{$\Gamma, \psi \Rightarrow \Delta$}
 \RightLabel{\footnotesize$L \vee$} 
 \BinaryInfC{$\Gamma, \phi \vee \psi \Rightarrow \Delta$}
 \DisplayProof
 &
 \AxiomC{$\Gamma \Rightarrow \phi , \psi, \Delta$}
 \RightLabel{\footnotesize$R \vee$} 
 \UnaryInfC{$\Gamma \Rightarrow \phi \vee \psi, \Delta$}
 \DisplayProof
 \\[3ex]
 \AxiomC{$\Gamma \Rightarrow \phi, \Delta$}
 \AxiomC{$\Gamma, \psi \Rightarrow \Delta$}
 \RightLabel{\footnotesize$L \to$} 
 \BinaryInfC{$\Gamma, \phi \to \psi \Rightarrow \Delta$}
 \DisplayProof
 &
 \AxiomC{$\Gamma, \phi \Rightarrow \psi, \Delta$}
 \RightLabel{\footnotesize$R \to$} 
 \UnaryInfC{$\Gamma \Rightarrow \phi \to \psi, \Delta$}
 \DisplayProof
\end{tabular}
\caption{The sequent calculus $\Gthc$. In \textit{Ax}, the formula $p$ must be atomic.}
\label{figgthc}
\end{center}
\end{figure}
\noindent We also use the following two explicit weakening rules:  
\begin{center}
\begin{tabular}{c c}
\AxiomC{$\Gamma \Rightarrow \Delta$}
\RightLabel{\footnotesize$Lw$} 
\UnaryInfC{$\Gamma, \phi \Rightarrow \Delta$}
\DisplayProof
 & \quad
 \AxiomC{$\Gamma \Rightarrow \Delta$}
 \RightLabel{\footnotesize$Rw$} 
 \UnaryInfC{$\Gamma \Rightarrow \phi, \Delta$}
 \DisplayProof
\end{tabular}
\end{center}
The system $\mathbf{G3cp}$ plus the weakening rules is denoted by $\GthW$. In each rule in $\GthW$, the multiset $\Gamma$ (resp., $\Delta$) is called the left (resp., right) \textit{context}, the formulas outside $\Gamma \cup \Delta$ are called the \emph{active formulas} of the rule and the only formula in the conclusion outside
$\Gamma \cup \Delta$ is called the \emph{main formula}. If $S$ is the conclusion of an instance of a rule $R$, we say that $R$ is {\em backward applicable} to $S$. \\

The modal rules by which we extend $\Gthc$ or $\GthW$ are given in Figure~\ref{figmodal}. For any such rule $(X)$ except $(EC)$, $(N)$ and $(NW)$, if we add it to $\GthW$ we denote the resulting system by $\mathbf{GX}$, and if we add $(N)$ to that system we get $\mathbf{GXN}$.
\begin{figure}[H]
\begin{center}
\begin{tabular}{c c c}
\def\defaultHypSeparation{\hskip .05in}
\AxiomC{$ \phi \Rightarrow \psi$} 
  \AxiomC{$ \psi \Rightarrow \phi$} 
   \RightLabel{\footnotesize$E$} 
 \BinaryInfC{$ \Box \phi \Rightarrow \Box \psi$}
 \DisplayProof
 &
 \def\defaultHypSeparation{\hskip .05in}
 \AxiomC{$\phi_1, \cdots , \phi_n \Rightarrow \psi$} 
  \AxiomC{$\psi \Rightarrow \phi_1 \;\;\;\;\; \cdots$}  
    \AxiomC{$\psi \Rightarrow \phi_n$} 
   \RightLabel{\footnotesize$EC$} 
 \TrinaryInfC{$\Sigma, \Box \phi_1, \cdots , \Box \phi_n \Rightarrow \Box \psi, \Lambda$}
 \DisplayProof
\end{tabular}
\end{center}
\begin{center}
\begin{tabular}{c c c c}
 \AxiomC{$ \phi \Rightarrow \psi$} 
  \RightLabel{\footnotesize$M$} 
 \UnaryInfC{$ \Box \phi \Rightarrow \Box \psi$}
 \DisplayProof
 &
  \AxiomC{$ \Rightarrow \psi$} 
   \RightLabel{\footnotesize$N$} 
 \UnaryInfC{$ \Rightarrow \Box \psi$}
 \DisplayProof
 &
  \AxiomC{$ \Rightarrow \psi$} 
   \RightLabel{\footnotesize$NW$} 
 \UnaryInfC{$\Sigma \Rightarrow \Box \psi, \Lambda$}
 \DisplayProof
 &
 \AxiomC{$\phi_1, \cdots, \phi_n \Rightarrow \psi$} 
   \RightLabel{\footnotesize$MC$} 
 \UnaryInfC{$\Box \phi_1, \cdots, \Box \phi_n \Rightarrow \Box \psi$}
 \DisplayProof
\end{tabular}
\end{center}
\caption{The modal rules. In $(EC)$ and $(MC)$, we must have $n \geq 1$.}
\label{figmodal}
\end{figure}
\noindent Note that $\mathbf{GMCN}$ is the usual system for the 
logic $\mathsf{K}$. 
If we add 
$(EC)$ to $\Gthc$, we get $\mathbf{GEC}$ and if we also add the rule $(NW)$, we get 
$\mathbf{GECN}$. Note that 
$\mathbf{GEC}$ and $\mathbf{GECN}$ have no explicit weakening rules.\linebreak
The systems $\mathbf{GEC}$ and $\mathbf{GECN}$ are introduced in \cite{Orlandelli}. The others are equivalent to the systems introduced in \cite{Orlandelli}. 
The only difference is that in our presentation, the weakening rules are explicitly present, while the extra context in the conclusion of the modal rules are omitted. We will present the systems as such for convenience in our later proofs. As the systems $\mathbf{GE}$, $\mathbf{GM}$, $\mathbf{GMC}$, $\mathbf{GEN}$ and $\mathbf{GMN}$ are equivalent to the systems introduced in \cite{Orlandelli}, they all admit the cut rule and the contraction rules. Moreover, the logics of these systems, i.e., the sets of formulas $\phi$ for which the systems prove $(\, \Rightarrow \phi)$ are the basic non-normal modal logics $\mathsf{E}$, $\mathsf{M}$, $\mathsf{MC}$, $\mathsf{EN}$ and $\mathsf{MN}$, respectively, introduced in the previous subsection. The logics of the systems $\mathbf{GEC}$ and $\mathbf{GECN}$ are the logics $\mathsf{EC}$ and $\mathsf{ECN}$, respectively \cite{Orlandelli}. \\

For the conditional logics $\mathsf{CE}, \mathsf{CM}, \mathsf{CMC}, \mathsf{CEN}$, and $\mathsf{CMN}$, their sequent calculi have not been studied before, as far as we know. In Section \ref{SequentSystems}, we introduce their corresponding sequent calculi, i.e., $\GCE$, $\GCM$, $\GCMC$,
$\GCEN$, and $\GCMN$ and prove that they enjoy the cut elimination theorem. However, the sequent calculi for the conditional logics $\mathsf{CK}$, $\mathsf{CKID}, \mathsf{CKCEM},$ and $\mathsf{CKCEMID}$ have already been studied in \cite{PattinsonSchroeder}. For the first, we reintroduce the system as $\mathbf{GCK}$ in Section \ref{SequentSystems}. For the others, we provide the rules depicted in Figure \ref{figconditional} that are equivalent to the ones introduced in \cite{PattinsonSchroeder}. The difference is that in our presentation, we use two-sided sequent calculi while the calculi used in \cite{PattinsonSchroeder} are  one-sided. Moreover, the weakening rules are explicitly present in our systems. Adding each of these rules to $\GthW$ will result in the sequent calculi $\mathbf{GCKID}$, $\mathbf{GCKCEM}$, and $\mathbf{GCKCEMID}$, respectively.
\begin{figure}[H]
    \begin{center}
\begin{tabular}{c}
\AxiomC{$\{\phi_0 \Rightarrow \phi_i \;\;\; , \;\;\; \phi_i \Rightarrow \phi_0\}_{i \in I}$} 
 \AxiomC{$\phi_0, \{\psi_i \}_{i \in I} \Rightarrow \psi_0$} 
   \RightLabel{\footnotesize$CKID$} 
 \BinaryInfC{$\{\phi_i \triangleright \psi_i\}_{i \in I} \Rightarrow \phi_0 \triangleright \psi_0$}
 \DisplayProof
\end{tabular}
\end{center}
\begin{center}
\begin{tabular}{c}
\AxiomC{$\{\phi_0 \Rightarrow \phi_r \;\;\; , \;\;\; \phi_r \Rightarrow \phi_0\}_{r \in I \cup J}$} 
 \AxiomC{$\{\psi_i\}_{i \in I} \Rightarrow \psi_{0}, \{\psi_j\}_{j \in J}$} 
   \RightLabel{\footnotesize$CKCEM$} 
 \BinaryInfC{$\{\phi_i \triangleright \psi_i\}_{i \in I} \Rightarrow \phi_{0} \triangleright \psi_{0}, \{\phi_j \triangleright \psi_j\}_{j \in J}$}
 \DisplayProof
\end{tabular}
\end{center}
\begin{center}
\begin{tabular}{c}
\AxiomC{$\{\phi_0 \Rightarrow \phi_r \;\;\; , \;\;\; \phi_r \Rightarrow \phi_0\}_{r \in I \cup J}$} 
 \AxiomC{$\phi_0, \{\psi_i\}_{i \in I} \Rightarrow  \psi_{0}, \{\psi_j\}_{j \in J}$} 
   \RightLabel{\footnotesize$CKCEMID$} 
 \BinaryInfC{$\{\phi_i \triangleright \psi_i\}_{i \in I} \Rightarrow \phi_{0} \triangleright \psi_{0}, \{\phi_j \triangleright \psi_j\}_{j \in J}$}
 \DisplayProof
\end{tabular}
\end{center}
    \caption{Conditional rules I}
    \label{figconditional}
\end{figure}

Here are some remarks about the rules introduced above. First, for any rule the weight of each premise is less than the weight of its conclusion. For the propositional rules or the weakening rule the claim is clear. For the modal and conditional rules, note that the weight of the sequent $ \Gamma, \Sigma \Rightarrow \Delta, \Lambda$ is less than the weight of $\Box \Gamma, \Sigma \Rightarrow \Box \Delta, \Lambda$ as long as $\Gamma \cup \Delta$ is non-empty. Similarly, for any finite multisets $I$ and $J$ and for any $M \subseteq I$ and $N \subseteq J$, it is clear that the weight of any subsequent of the sequent 
\[
\{\phi_i, \psi_i\}_{i \in M}, \{\theta_j, \chi_j\}_{j \in N},  \Sigma \Rightarrow \{ \phi_{i}, \psi_{i}\}_{i \in I-M}, \{\theta_j, \chi_j\}_{j \in J-N}, \Lambda
\]
is less than the weight of the sequent
$\{ \phi_{i} \, \triangleright \, \psi_i\}_{i \in I}, \Sigma \Rightarrow \{\theta_j \, \triangleright \, \chi_j\}_{j \in J}, \Lambda$, as long as $I \cup J$ is non-empty. Secondly, note that 
in any rule in $\GthW$, if we add a multiset, both to the antecedent (succedent) 
of the premises and to the antecedent (succedent) 
of the conclusion, the result remains an instance of the 
rule. 
We call this property the \emph{context extension property}. 
Conversely, if a multiset is a sub-multiset of the left (right) context of the rule, then if we eliminate this multiset both from the premises and the conclusion, the result remains an instance of the
rule. 
We call this property the \emph{context restriction property}. 
Thirdly, 
for any rule in $\GthW$ and any $\circ \in \{+, -\}$, if the main formula $\phi$ is in the antecedent, 
then for any active formula $\alpha$ in the 
antecedent of a premise and any active formula $\beta$ 
in the succedent of a premise, we have $V^{\circ}(\alpha) \cup V^{\diamond}(\beta) \subseteq V^{\circ}(\phi)$, and if 
$\phi$ is in the succedent, we have $V^{\diamond}(\alpha) \cup V^{\circ}(\beta) \subseteq V^{\circ}(\phi)$ (note the use of $\circ$ and $\diamond$). We call this property, the \emph{variable preserving property}. 
As a consequence of this property 
for the rule
\begin{center}
\begin{tabular}{c c c}
 \AxiomC{$S_1 \; \cdots \; S_n$}  
 \UnaryInfC{$S$}
 \DisplayProof
\end{tabular}
\end{center}
in $\GthW$, we have $\bigcup_{i=1}^n V^{\circ}(S_i) \subseteq V^{\circ}(S)$, for any $\circ \in \{+, -\}$.

\section{Sequent Calculi for Non-normal Conditional Logics} \label{SequentSystems}
In this section, we introduce sequent calculi for the six conditional logics $\mathsf{CE}, \mathsf{CM}, \mathsf{CMC}, \mathsf{CEN}, \mathsf{CMN},$ and $\mathsf{CK}$. Then, we prove that in all of these six calculi the cut rule is admissible and conclude that the presented systems are sound and complete for their corresponding logics. The system for $\mathsf{CK}$ and its cut elimination proof is not new and introduced before in \cite{PattinsonSchroeder}. However, we include $\mathsf{CK}$ for the completeness' sake.\\

\noindent To present the systems, consider the following set of conditional rules:

\begin{figure}[H]
\begin{center}
\begin{tabular}{c c c}
\AxiomC{$\phi_0 \Rightarrow \phi_1 \;\;\;  \;\;\; \phi_1 \Rightarrow \phi_0$} 
 \AxiomC{$\psi_0 \Rightarrow \psi_1 \;\;\;  \;\;\; \psi_1 \Rightarrow \psi_0$} 
   \RightLabel{\footnotesize$CE$} 
 \BinaryInfC{$\phi_1 \triangleright \psi_1 \Rightarrow \phi_0 \triangleright \psi_0$}
 \DisplayProof
\end{tabular}
\end{center}

\begin{center}
\begin{tabular}{c c c}
\AxiomC{$\phi_0 \Rightarrow \phi_1 \;\;\;  \;\;\; \phi_1 \Rightarrow \phi_0$} 
 \AxiomC{$\psi_1 \Rightarrow \psi_0$} 
   \RightLabel{\footnotesize$CM$} 
 \BinaryInfC{$\phi_1 \triangleright \psi_1 \Rightarrow \phi_0 \triangleright \psi_0$}
 \DisplayProof
\end{tabular}
\end{center}

\begin{center}
\begin{tabular}{c c c}
\AxiomC{$\{\phi_0 \Rightarrow \phi_i \;\;\; , \;\;\; \phi_i \Rightarrow \phi_0\}_{1 \leq i \leq n}$} 
 \AxiomC{$\psi_1, \cdots, \psi_n \Rightarrow \psi_0$} 
   \RightLabel{\footnotesize$CMC$ $(n \geq 1)$} 
 \BinaryInfC{$\phi_1 \triangleright \psi_1, \cdots, \phi_n \triangleright \psi_n \Rightarrow \phi_0 \triangleright \psi_0$}
 \DisplayProof
\end{tabular}
\end{center}

\begin{center}
\begin{tabular}{c c c}
 \AxiomC{$ \Rightarrow \psi_0$} 
   \RightLabel{\footnotesize$CN$} 
 \UnaryInfC{$ \Rightarrow \phi_0 \triangleright \psi_0$}
 \DisplayProof
\end{tabular}
\end{center}
    \caption{Conditional rules II}
    \label{figconditionalII}
\end{figure}
For any rule $(X)$, from Figure~\ref{figconditionalII}, except $(CN)$, let  $\mathbf{GX}$ denote the calculus $\GthW$ extended by $(X)$. If $(CN)$ is added to $\mathbf{GX}$, the resulting calculus is denoted by $\mathbf{GXN}$. Thus, we have defined six sequent calculi, $\GCE$, $\GCM$, $\GCMC$, $\GCEN$, $\GCMN$, and $\GCMCN$, where the last is the system introduced in \cite{PattinsonSchroeder} for $\mathsf{CK}$. We denote this system by $\mathbf{GCK}$. Our next task is to show that in all of these systems, the cut rule is admissible. \\ 

Let $G$ be any of the six aforementioned calculi. We show that the cut rule is admissible in $G$ by induction on the weight and a subinduction on the level of cuts in a proof. The weight of an application of the cut rule is the weight of the cut formula. The {\em level} of a cut in a derivation is the sum of the heights of its two premisses, where 
the {\em height} of a derivation is the length of its longest branch, where branches consisting of one node are considered to have height 1. 
The following lemma records the two admissible rules that the cut elimination process usually needs:
\begin{lemma} \label{AdmissibleRules} The following rules are admissible in $G$: 
\begin{description}
\item[$(i)$]
Contraction, i.e., 
\begin{eqnarray*}
 G \vdash \Ga,\phi,\phi \seq \De & \text{ implies } &G \vdash \Ga,\phi \seq \De \\
 G \vdash \Ga \seq \phi,\phi,\De & \text{ implies } & G \vdash \Ga \seq \phi, \De.
\end{eqnarray*}
\item[$(ii)$]
$\bot$-elimination, i.e., 
\begin{eqnarray*}
 G \vdash \Ga \seq \bot, \De & \text{ implies } & G \vdash \Ga \seq \De \\
\end{eqnarray*}
\end{description}
\end{lemma}
\begin{proof}
The proof uses a simple induction on the proof length as explained in \cite{Troelstra} and is quite routine. Hence, we will skip the details here.
\end{proof}

\begin{theorem}({\it Cut-elimination}) \label{thmcutadmll}
The following cut rule is admissible in $G$.
\begin{center}
 \AxiomC{$\Ga_1 \seq \phi,\De_1$}
 \AxiomC{$\phi,\Ga_2 \seq \De_2$}
 \RightLabel{\footnotesize \it Cut}
 \BinaryInfC{$\Ga_1, \Ga_2 \seq \De_1,\De_2$}
 \DisplayProof
\end{center}
\end{theorem}
\begin{proof}
Following the
cut elimination proof for $\Gthc$ in \cite{Troelstra}, we successively eliminate cuts from proofs, always considering those cuts that have no cuts above them, the {\em topmost} cuts. For this, it suffices to show that for any cut-free proofs $\cald_1$ and $\cald_2$, the following proof $\cald$ of the sequent $S=(\Ga_1, \Ga_2 \seq \De_1,\De_2)$ can be transformed into a cut-free proof $\cald'$ of the same endsequent: 
\begin{center}
 \AxiomC{$\cald_1$} \noLine
 \UnaryInfC{$\Ga_1 \seq \phi,\De_1$}
 \AxiomC{$\cald_2$} \noLine
 \UnaryInfC{$\phi,\Ga_2 \seq \De_2$}
 \RightLabel{\footnotesize \it Cut}
 \BinaryInfC{$\Ga_1, \Ga_2 \seq \De_1,\De_2$}
 \DisplayProof
\end{center}
This is proved by induction on the weight of the cut formula, with a subinduction on the level of the cut. We use Lemma \ref{AdmissibleRules} implicitly at various places. \\
There are three possibilities: (1) at least one of the premises is an axiom, (2) both premises are not axioms and the cut formula is not principal in at least one of the premises, (3) the cut formula is principal in both premises, which are not axioms. \\
We only investigate case 3, and only the case where the last inference of the premises of the cut rule are instances of the conditional rules. The rest of the proof is completely similar to the proof of cut-elimination for $\Gthc$ in \cite{Troelstra}, using the fact that contraction and $\bot$-elimination are admissible in $G$, see Lemma \ref{AdmissibleRules}.\\
We distinguish by cases according to the conditional rules that the last inference of the premises are an instance of. To simplify the proof, in the rest of this proof, we use the following notation:

\begin{center}
 \small
 \AxiomC{$\cald^x$} \noLine
 \UnaryInfC{$\phi_i \Ifff \phi_j$} 
 \DisplayProof
is short for 
 \AxiomC{$\cald^x_{ij}$} \noLine
 \UnaryInfC{$\phi_i \seq \phi_j$} 
 \DisplayProof
 \AxiomC{$\cald^x_{ji}$} \noLine\
 \UnaryInfC{$\phi_j \seq \phi_i$} 
 \DisplayProof
\end{center}
We also make the convention that a {\em cut-free proof of $\phi\Ifff\psi$} means that both $\phi\seq\psi$ and $\psi\seq\phi$ have cut-free proofs. \\

Suppose that the last inference of the premises of the cut rule are instances of $(CE)$. Thus, the proof has the following form:
\begin{center}
 \small 
 \AxiomC{$\cald^0$} \noLine
 \UnaryInfC{$\phi_0\Ifff\phi_1$}
 \AxiomC{$\cald^1$} \noLine
 \UnaryInfC{$\psi_0\Ifff\psi_1$}
 \RightLabel{\footnotesize $CE$}
 \BinaryInfC{$\phi_1\triangleright\psi_1 \seq \phi_0 \triangleright \psi_0$}
 \AxiomC{$\cald^2$} \noLine
 \UnaryInfC{$\phi_2\Ifff\phi_0$}
 \AxiomC{$\cald^3$} \noLine
 \UnaryInfC{$\psi_2\Ifff\psi_0$}
 \RightLabel{\footnotesize $CE$}
 \BinaryInfC{$\phi_0\triangleright\psi_0 \seq \phi_2 \triangleright \psi_2$}
 \RightLabel{\footnotesize \it Cut}
 \BinaryInfC{$\phi_1 \triangleright \psi_1 \seq \phi_2 \triangleright \psi_2$}
 \DisplayProof
\end{center}
Consider the following proof of $\phi_1\seq \phi_2$:
\begin{center}
 \small 
 \AxiomC{$\cald^0_{10}$} \noLine
 \UnaryInfC{$\phi_1\seq\phi_0$}
 \AxiomC{$\cald^2_{02}$} \noLine
 \UnaryInfC{$\phi_0\seq\phi_2$}
 \RightLabel{\footnotesize\it Cut}
 \BinaryInfC{$\phi_1 \seq \phi_2$}
 \DisplayProof
\end{center}
The cut has lower weight and thus by the induction hypothesis, there exists a cut-free proof $\cald_{12}$ of $\phi_1\seq\phi_2$. Likewise, there is a cut-free proof of $\phi_2\seq\phi_1$. Therefore, we have a cut free proof of $\phi_1\Ifff\phi_2$. Similarly, there is a cut-free proof of $\psi_1\Ifff\psi_2$. An application of $(CE)$ results in a cut-free proof of the sequent $\phi_1\triangleright\psi_1 \seq \phi_2\triangleright\psi_2$.\\
Suppose that the last inferences of the premises of the cut are instances of $(CM)$. Thus, the proof has the following form, where
\begin{center}
 \small 
 \AxiomC{$\cald^0$} \noLine
 \UnaryInfC{$\phi_0\Ifff\phi_1$}
 \AxiomC{$\cald^1$} \noLine
 \UnaryInfC{$\psi_1\seq\psi_0$}
 \RightLabel{\footnotesize$CM$}
 \BinaryInfC{$\phi_1\triangleright\psi_1 \seq \phi_0 \triangleright \psi_0$}
 \AxiomC{$\cald^2$} \noLine
 \UnaryInfC{$\phi_2 \Ifff \phi_0$}
 \AxiomC{$\cald^3$} \noLine
 \UnaryInfC{$\psi_0 \seq \psi_2$}
 \RightLabel{\footnotesize$CM$}
 \BinaryInfC{$\phi_0\triangleright\psi_0 \seq \phi_2 \triangleright \psi_2$}
 \RightLabel{\footnotesize \it Cut}
 \BinaryInfC{$\phi_1 \triangleright \psi_1 \seq \phi_2 \triangleright \psi_2$}
 \DisplayProof
\end{center}
As in the case of $(CE)$, it follows that there is a cut-free proof of $\phi_1\Ifff\phi_2$. Also, a cut on $\psi_0$ after $\cald_1$ and $\cald_3$ shows that there is a cut-free proof of $\psi_1 \seq \psi_2$. An application of $(CM)$ then gives a cut-free proof of $\phi_1 \triangleright \psi_1 \seq \phi_2 \triangleright \psi_2$.\\
Suppose that the last inferences above the cut are instances of the rule $(CMC)$. Thus, the proof has the following form, where $\Ga$ and $\Ga'$ abbreviate $\{\phi_1 \triangleright \psi_1,\dots,\phi_n\triangleright\psi_n\}$ and $\{\phi_1' \triangleright \psi_1',\dots,\phi_m'\triangleright\psi_m'\}$, respectively: 
\begin{center}
 \small 
 \AxiomC{$\cald_1$} \noLine
 \UnaryInfC{$\Ga \seq \phi_0 \triangleright \psi_0$}
 \AxiomC{$\cald_2$} \noLine
 \UnaryInfC{$\phi_0 \triangleright \psi_0,\Ga'\seq \phi_0' \triangleright \psi_0'$}
 \RightLabel{\footnotesize \it Cut}
 \BinaryInfC{$\Ga,\Ga'\seq \phi_0' \triangleright \psi_0'$}
 \DisplayProof
\end{center}
Here $\cald_1$ and $\cald_2$ are the following two derivations: 
\begin{center}
 \small 
 \AxiomC{$\cald^1_1$} \noLine
 \UnaryInfC{$\phi_0\Ifff\phi_1$}
 \AxiomC{$\dots$} 
 \AxiomC{$\cald^1_n$} \noLine
 \UnaryInfC{$\phi_0\Ifff\phi_n$}
 \AxiomC{$\cald_1'$} \noLine
 \UnaryInfC{$\psi_1,\dots,\psi_n \seq \psi_0$}
 \RightLabel{\footnotesize \it CMC}
 \LeftLabel{$\cald_1$:\ \ \ \ \ }
 \QuaternaryInfC{$\Ga \seq \phi_0 \triangleright \psi_0$}
 \DisplayProof
\end{center}
\begin{center}
 \small 
 \AxiomC{$\cald^2_1$} \noLine
 \UnaryInfC{$\phi_0'\Ifff\phi_1'$}
 \AxiomC{$\dots$} 
 \AxiomC{$\cald^2_n$} \noLine
 \UnaryInfC{$\phi_0'\Ifff\phi_m'$}
 \AxiomC{$\cald_0$} \noLine
 \UnaryInfC{$\phi_0'\Ifff \phi_0$}
 \AxiomC{$\cald_2'$} \noLine
 \UnaryInfC{$\psi_0,\psi_1',\dots,\psi_m' \seq \psi_0'$} 
 \RightLabel{\footnotesize \it CMC}
 \LeftLabel{$\cald_2$:\ \ \ \ \ }
 \QuinaryInfC{$\phi_0 \triangleright \psi_0,\Ga'\seq \phi_0' \triangleright \psi_0'$}
 \DisplayProof
\end{center}
As before, it is easy to prove the existence of a cut-free proof for the sequents $\phi'_0 \Leftrightarrow \phi_i$ and $\phi'_0 \Leftrightarrow \phi'_j$, for any $1 \leq i \leq n$ and $1 \leq j \leq m$. Now, consider the following proof:
\begin{center}
 \small 
 \AxiomC{$\cald_1'$} \noLine
 \UnaryInfC{$\psi_1,\dots,\psi_n \seq \psi_0$}
 \AxiomC{$\cald_2'$} \noLine
 \UnaryInfC{$\psi_0,\psi_1',\dots,\psi_m' \seq \psi'_0$}
  \RightLabel{\footnotesize \it Cut}
 \BinaryInfC{$\psi_1,\dots,\psi_n,\psi_1',\dots,\psi_m' \seq \psi'_0$}
 \DisplayProof
\end{center}
Since its cut is of lower weight, there exists a cut-free proof of its endsequent. Applying the following instance of the rule $(CMC)$, we reach the cut-free proof we are looking for:
\begin{center}
 \small 
 \AxiomC{$\{\phi'_0\Ifff\phi_i\}_{i=1}^n$}
 \AxiomC{$\{\phi_0'\Ifff\phi_j'\}_{j=1}^m$}
 \AxiomC{$\psi_1,\dots,\psi_n,\psi_1',\dots,\psi_m' \seq \psi'_0$}
 \TrinaryInfC{$\Ga,\Ga'\seq \phi_0' \triangleright \psi_0'$}
 \DisplayProof
\end{center}
Finally, for the systems in the form $\mathbf{GXN}$ that contain the rule $(CN)$, note that the last inference of the right branch above the cut cannot be an instance of $(CN)$, as any of its instances has an empty antecedent. Hence, there are two cases to consider. The case that both last inferences are some instances of $(X)$ has been treated above. Therefore, only the case that the left premise is the conclusion of $(CN)$ and the right premise is the conclusion of $(X)$ remains. We only spell out the details for the case $X=CMC$. The others are similar. By assumption, the proof has the following form, where $\Ga$ abbreviates $\{\phi_2 \triangleright \psi_2,\dots,\phi_n\triangleright\psi_n\}$.
\begin{center}
 \small 
 \AxiomC{$\cald_0$} \noLine
 \UnaryInfC{$\ \seq \psi_1$}
  \RightLabel{\footnotesize \it NC }
 \UnaryInfC{$\ \seq \phi_1 \triangleright \psi_1$}
 \AxiomC{$\cald_1^1$} \noLine
 \UnaryInfC{$\phi_0\Ifff \phi_1$}
 \AxiomC{\dots} 
 \AxiomC{$\cald_1^n$} \noLine
 \UnaryInfC{$\phi_0\Ifff \phi_n$}
 \AxiomC{$\cald_2$}\noLine
 \UnaryInfC{$\psi_1,\dots,\psi_n \seq \psi_0$}
 \RightLabel{\footnotesize \it CMC }
 \QuaternaryInfC{$\phi_1\triangleright \psi_1,\Ga \seq \phi_0 \triangleright \psi_0$} 
 \RightLabel{\footnotesize\it Cut}
 \BinaryInfC{$\Ga \seq \phi_0 \triangleright \psi_0$}
 \DisplayProof
\end{center}
The following proof of the same endsequent is of lower weight:
\begin{center}
 \small 
 \AxiomC{$\cald_1^2$} \noLine
 \UnaryInfC{$\phi_0\Ifff \phi_2$}
 \AxiomC{\dots} 
 \AxiomC{$\cald_1^n$} \noLine
 \UnaryInfC{$\phi_0\Ifff \phi_n$}
 \AxiomC{$\cald_0$} \noLine
 \UnaryInfC{$\ \seq \psi_1$}
 \AxiomC{$\cald_2$}\noLine
 \UnaryInfC{$\psi_1,\dots,\psi_n \seq \psi_0$}
 \RightLabel{\footnotesize \it Cut}
 \BinaryInfC{$\psi_2,\dots,\psi_n \seq \psi_0$}
 \RightLabel{\footnotesize \it CMC }
 \QuaternaryInfC{$\Ga \seq \phi_0 \triangleright \psi_0$}
 \DisplayProof
\end{center}
that completes the proof.
\end{proof} 

\begin{corollary}
The system $\mathbf{GCE}$, $\mathbf{GCM}$, $\mathbf{GCEN}$, $\mathbf{GCMN}$, $\mathbf{GCMC}$ and $\mathbf{GCK}$ are sound and complete for the logics $\mathsf{CE}$, $\mathsf{CM}$, $\mathsf{CEN}$, $\mathsf{CMN}$, $\mathsf{CMC}$ and $\mathsf{CK}$, respectively.
\end{corollary}
\begin{proof}
By Theorem \ref{thmcutadmll}, cut is admissible in all the systems. Using this fact, it is easy to see that each rule in any of these systems can simulate its corresponding rule or axiom in the corresponding logic and vice versa.
\end{proof}

\section{Two Flavours of Uniform Interpolation}\label{ULIPSection}
In this section, we will establish two different flavours of uniform interpolation for some of the modal and conditional logics introduced in Section \ref{Preliminaries}. First, we prove uniform Lyndon interpolation property for the logics $\mathsf{E}$, $\mathsf{M}$, $\mathsf{MC}$, $\mathsf{EN}$, $\mathsf{MN}$, $\mathsf{K}$, their conditional versions $\mathsf{CE}$, $\mathsf{CM}$, $\mathsf{CMC}$, $\mathsf{CEN}$, $\mathsf{CMN}$, and $\mathsf{CK}$, in addition to the conditional logic $\mathsf{CKID}$. Then, we move to uniform interpolation property to show that the logics $\mathsf{CKCEM}$ and $\mathsf{CKCEMID}$ enjoy UIP.\\
To explain our general strategy, we first need to extend these flavours of interpolation from the logics to the sequent calculi. Since 
all the logics considered in this paper are classical, we 
only define the universal quantifier, as the existential quantifier is constructed 
by the universal quantifier and negation. 

\begin{definition} \label{DfnUniformInterpolationSeq} 
Let $G$  be one of the sequent calculi introduced in this paper. Then,
$G$ has \emph{uniform Lyndon interpolation property (ULIP)}, if for any sequent $S$, any atom $p$ and any $\circ \in \{+, -\}$, there exists a formula $\Ap S$ such that:
\begin{description}
\item[$(var)$]
$\Ap S$ is $p^{\circ}$-free and $V^{\dagger}(\Ap S) \subseteq V^{\dagger}(S)$, for any $\dagger \in \{+, -\}$,
\item[$(i)$]
$S \cdot (\Ap S \Rightarrow)$ is derivable in $G$,
\item[$(ii)$]
for any sequent $\Gamma \Rightarrow \Delta$ such that $p \notin V^{\diamond}(\Gamma \Rightarrow \Delta)$, if $S \cdot (\Gamma \Rightarrow \Delta)$ is derivable in $G$ then $(\Gamma \Rightarrow \Ap S, \Delta)$ is derivable in $G$.
\end{description}
$\Ap S $ is called a 
{\em uniform $\A_p^{\circ}$-interpolant of $S$} in $G$. 
For any set of rules $\mathcal{R}$ of $G$, a formula $\Apr S$ is called a 
{\em uniform $\A_p^{\circ}$-interpolant of $S$ with respect to $\mathcal{R}$}, 
if it satisfies the 
conditions $(var)$ and $(i)$, when $\Ap S$ is replaced by $\Apr S$, and: 
\begin{description}
\item[$(ii')$]
for any sequent $\Gamma \Rightarrow \Delta$ such that $p \notin V^{\diamond}(\Gamma \Rightarrow \Delta)$, if there is a derivation of $S \cdot (\Gamma \Rightarrow \Delta)$ in $G$ whose last inference rule is an instance of a rule in $\mathcal{R}$, then $(\Gamma \Rightarrow \Apr S, \Delta)$ is derivable in $G$.
\end{description}
The sequent calculus $G$ has \emph{uniform interpolation property (UIP)} if it has all the above properties, omitting the superscripts $\circ, \diamond, \dagger \in \{+, -\}$, everywhere. Specifically, for any set of rules $\mathcal{R}$ of $G$, a formula $\forall_{\mathcal R} p S$, called a 
{\em uniform $\A_p$-interpolant of $S$ with respect to $\mathcal{R}$}, if it satisfies $(var)$, $(i)$, and $(ii')$, omitting $\circ, \diamond, \dagger \in \{+, -\}$ everywhere.
\end{definition}

\begin{remark}
As the formula $\Ap S$ is provably unique, using the functional notation of writing $\Ap S$ as a function with the arguments $\circ \in \{+, -\}$, $p$ and $S$ is allowed. The same does not hold for $\Apr S$. However, as there is no risk of confusion and we will be specific about the construction of the formula $\Apr S$, we will also use the functional notation in this case. The situation with $\forall p S$ and $\forall_{\mathcal R} p S$ is similar.
\end{remark}

The following theorem connects ULIP and UIP for sequent calculi to the original version defined for logics, as we expect:
\begin{theorem} \label{SequentImpliesLogic}
Let $G$ be one of the sequent calculi introduced in this paper and $L$ be its logic. Then, 
$G$ has ULIP (resp., UIP) 
iff $L$ has ULIP (resp., UIP). 
\end{theorem}
\begin{proof}
We only provide the sketch of the proof for ULIP. The proof for the other case is identical and only requires omitting the superscripts $\circ, \diamond \in \{+, -\}$. If $G$ has ULIP, 
set $\Ap A=\Ap (\, \Rightarrow A)$ and $\Ep A=\neg \Apd \neg A$. 
Conversely, if $L$ has ULIP, 
set $\Ap (\Gamma \Rightarrow \Delta)=\Ap (\bigwedge \Gamma \to \bigvee \Delta)$. 
\end{proof}

Our strategy to prove ULIP or UIP for a logic is to prove the same property for its corresponding sequent calculus. We will explain the strategy for the harder case of ULIP. However, the proofs also work for UIP and the general constructions must be read as the argument towards proving UIP, as well. We will use this general strategy and the corresponding theorems in Subsection \ref{ULIPSubsection} to prove ULIP for the logics mentioned at the beginning of this section and in Subsection \ref{SecUIP} to prove UIP for the logics $\mathsf{CKCEM}$ and $\mathsf{CKCEMID}$.\\
To explain the main strategy, let $G$ be the system for one of the logics mentioned at the beginning of this section and recall that the backward application of any of the rules in $G$ decreases the weight of the sequent. Using this property and recursion on the weight of the sequents, for any given sequent $S= (\Gamma \Rightarrow \Delta)$, any atom $p$ and any $\circ \in \{+, -\}$, we first define a $p^{\circ}$-free formula $\Ap S$ and then by induction on the weight of $S$, we 
prove that $\Ap S$ meets 
the conditions in Definition \ref{DfnUniformInterpolationSeq}. Towards that end, 
both in the definition of $\Ap S$ and in the proof of its properties, we must address all the rules of the system $G$, one by one. To make the presentation uniform, modular, and more clear, we divide the rules of $G$ into two families: the rules of $\GthW$ and the modal or the conditional rules specific for $G$. The rules in the first class has one of the following forms:

\begin{center}
 \begin{tabular}{c c}
 \AxiomC{$\{ \Gamma, \bar{\phi}_i \Rightarrow \bar{\psi}_i, \Delta \}_i$}
 \UnaryInfC{$\Gamma, \phi \Rightarrow \Delta$}
 \DisplayProof \;\;\;
 &
 \AxiomC{$\{ \Gamma, \bar{\phi}_i \Rightarrow \bar{\psi}_i, \Delta \}_i$}
 \UnaryInfC{$\Gamma \Rightarrow \phi, \Delta$}
 \DisplayProof
\end{tabular}
\end{center}
where $\Gamma$ and $\Delta$ are free for all multiset substitutions, and $\bar{\phi}_i$'s and $\bar{\psi}_i$'s are multisets of formulas (possibly empty). The rules 
have the \emph{variable preserving condition}, i.e., given $\circ \in \{+, -\}$, for the left rule 
\[
\bigcup_i \bigcup_{\theta \in \bar{\phi}_i} V^{\circ}(\theta) \cup \bigcup_i \bigcup_{\theta \in \bar{\psi}_i} V^{\diamond}(\theta) \subseteq V^{\circ}(\phi),
\]
and for the right one
\[
\bigcup_i \bigcup_{\theta \in \bar{\phi}_i} V^{\diamond}(\theta) \cup \bigcup_i \bigcup_{\theta \in \bar{\psi}_i} V^{\circ}(\theta) \subseteq V^{\circ}(\phi).
\]
\noindent Rather than addressing each 
rule in $\GthW$, we simply 
address these two forms that cover all the rules in $\GthW$.

\begin{lemma}\label{Axiom}
For any sequent $S$, atom $p$ and $\circ \in \{+, -\}$, a uniform $\A^{\circ}_p$-interpolant (resp., a uniform $\A_p$-interpolant) of $S$ with respect to the set of all axioms of $G$ exists. 
\end{lemma}
\begin{proof}
We only prove the existence of a $\A^{\circ}_p$-interpolant. The other case is similar. It is just enough to omit all the occurrences of $\circ$, $\diamond$ and $\dagger$ in the proof.
Let us define the formula $\Apax S$ in the following way: if $S$ is provable, define it  
as $\top$, otherwise, define 
it as the disjunction of all $p^{\circ}$-free formulas in $S^s$ and the negation of all $p^{\diamond}$-free formulas in $S^a$. We show that  $\Apax S$ is the uniform $\A_p^{\circ}$-interpolant of $S$ with respect to the set $\mathcal{A}$ of axioms of $G$. 
It is easy to see that 
$\Apax S$ is $p^{\circ}$-free, $V^{\dagger}(\Apax S) \subseteq V^{\dagger}(S)$, for 
$\dagger \in \{+, -\}$ and $S \cdot (\Apax S \Rightarrow )$ is provable in $G$. To prove the condition $(ii')$ in Definition \ref{DfnUniformInterpolationSeq}, if $S$ is provable, then as $\Apax S=\top$, we have $\bar{C} \Rightarrow \forall^{\circ}_{\mathcal{A}} p S, \bar{D}$. If $S$ is not provable, then let $S \cdot (\bar{C} \Rightarrow \bar{D})$ be an axiom. There are two cases to consider. First, if $S \cdot (\bar{C} \Rightarrow \bar{D})$ is in the form $\Gamma, q \Rightarrow q, \Delta$, where $q$ is an atomic formula. Then, if  $q \notin \bar{C}$ and $q \notin \bar{D}$, we have $q \in \Gamma \cap \Delta$ and hence the sequent $S$ is 
provable which contradicts our assumption. Therefore, either $q \in \bar{C}$ or $q \in \bar{D}$. If $q \in \bar{C} \cap \bar{D}$, then $\bar{C} \Rightarrow \forall^{\circ}_{\mathcal{A}} p S, \bar{D}$ is provable. Hence, we assume either $q \in \bar{C}$ and $q \notin \bar{D}$ or $q \notin \bar{C}$ and $q \in \bar{D}$. In the first case, if $q \in \bar{C}$, it is $p^{\circ}$-free and since it occurs in $\Delta$, it is a disjunct in $\Apax S$. Hence, $\bar{C} \Rightarrow \Apax S, \bar{D}$ is provable. In the second case, if $q \in \bar{D}$, it is $p^{\diamond}$-free and as $q \in \Gamma$, its negation occurs in $\Apax S$. Therefore $\bar{C} \Rightarrow \forall^{\circ}_{\mathcal{A}} p S, \bar{D}$ is provable.\\
If $S \cdot (\bar{C} \Rightarrow \bar{D})$ is in the form $\Gamma, \bot \Rightarrow \Delta$, then $\bot \in \bar{C}$, because otherwise, $\bot \in \Gamma$ and hence $S$ will be provable. Now, since $\bot \in \bar{C}$, we have $\bar{C} \Rightarrow \forall^{\circ}_{\mathcal{A}} p S, \bar{D}$.
\end{proof}

\begin{definition}\label{DfnMULIP}
Let $\mathcal{U}^{\circ}_p(S)$ be the statement that ``all sequents lower than $S$ have uniform $\A^{\circ}_p$-interpolants." 
A calculus $G$ has \emph{modal uniform Lyndon interpolation property} \emph{(MULIP)}
if for any sequent $S$, 
atom $p$, and 
$\circ \in \{+, -\}$, there exists a formula $\Apm S$ such that if $\mathcal{U}^{\circ}_p(S)$, then $\Apm S$ is a uniform $\A^{\circ}_p$-interpolant for $S$ with respect to the set $\mathcal{M}$ of modal or conditional rules of $G$. In a similar way, \emph{modal uniform interpolation property (MUIP)} for a calculus $G$ can be defined by omitting the superscript $\circ \in \{+, -\}$ everywhere. 
\end{definition}

Note that in Definition \ref{DfnMULIP}, the adjective modal in ``modal uniform Lyndon interpolation property" and ``modal uniform interpolation property" refers to any non-propositional operator in the language and hence covers the conditionals, as well. 

\begin{theorem}\label{GeneralArgument}
If a sequent calculus $G$ has MULIP (resp., MUIP), then it has ULIP (resp., UIP).
\end{theorem}
\begin{proof}
We only prove the ULIP case. The other case is similar. It is just enough to omit all the occurrences of $\circ$, $\diamond$ and $\dagger$ in the proof.\\
Define the formula $\Ap S$ by recursion on the weight of $S$:
 if $S$ is provable define it 
 as $\top$, otherwise, define it as: 
\[
\bigvee \limits_{R} (\bigwedge \limits_i \Ap S_{i})
 \vee (\Apax S) \vee (\Apm S)
\]
where the first disjunction is over all rules $R$ in $\GthW$ backward applicable to $S$, where $S$ is the consequence 
and $S_{i}$'s are the premises. 
$\Apax S$ is a 
uniform $\A^{\circ}_p$-interpolant of $S$ with respect to the set of 
axioms of $G$ that Lemma \ref{Axiom} provides. 
$\Apm S$ is the formula that 
MULIP provides. 
To prove that 
$\Ap S$ is a  $\A^{\circ}_p$-interpolant for $S$, we use induction on the weight of $S$ to prove:  
\begin{description}
\item[$(var)$] 
$\Ap S$ is $p^{\circ}$-free and $V^{\dagger}(\Ap S) \subseteq V^{\dagger}(S)$, for any $\dagger \in \{+, -\}$,
\item[$(i)$] 
$S \cdot (\Ap S \Rightarrow )$ is provable in $G$,
\item[$(ii)$]
for any $p^{\diamond}$-free sequent $\bar{C} \Rightarrow \bar{D}$, if $S \cdot (\bar{C} \Rightarrow \bar{D})$ is derivable in $G$ then $\bar{C} \Rightarrow \Ap S, \bar{D}$ is derivable in $G$.
\end{description}
By induction hypothesis, $(var)$, $(i)$, and $(ii)$ hold for all sequents $T$ lower than $S$. Now, $(var)$ also holds for $\Ap S$, because both $\forall^{\circ}_{\mathcal{A}} p S$ and $\Apm S$ satisfy $(var)$ and all rules in $\GthW$ have the variable preserving property.\\
To prove $(i)$, it is enough to show that the following are provable in $G$:
\[
S \cdot (\bigwedge \limits_i \Ap S_{i} \Rightarrow \,) \quad (1)\; , \quad
S \cdot (\Apax S \Rightarrow \,) \quad (2)  \;, \quad S \cdot (\Apm S \Rightarrow \,) \quad (3).
\]
Sequent (3) is provable by induction hypothesis and the assumption that $G$ has MULIP. 
Sequent (2)  is proved in Lemma \ref{Axiom}. For the sequent (1), assume that the rule $R$ of $\GthW$ is backward applicable 
to $S$, i.e., the premises of $R$ are $S_i$'s and its conclusion $S$.
As $S_i$'s are lower than $S$, by induction hypothesis we have $S_i \cdot (\Ap S_{i} \Rightarrow \,)$. Therefore, by weakening, we have $S_i \cdot (\{\Ap S_{i}\}_i \Rightarrow \, )$. Since any rule in $\GthW$ has the context extension property, we can add $\{\Ap S_{i}\}_i$ to the antecedent of both premises and conclusion and by the rule itself, we have $S \cdot (\{\Ap S_{i}\}_i \Rightarrow \, )$ and hence $S \cdot (\bigwedge_i \Ap S_{i} \Rightarrow \, )$.\\
For $(ii)$, we use 
induction on the length of the proof of $S \cdot (\bar{C} \Rightarrow \bar{D})$. 
Let $S \cdot (\bar{C} \Rightarrow \bar{D})$ be 
derivable in $G$. If it is an axiom, 
we have $\bar{C} \Rightarrow \bar{D}, \Apax S$ by Lemma \ref{Axiom}, and hence $\bar{C} \Rightarrow \bar{D}, \Ap S$. If the last rule is a rule in $\GthW$ of the form:
\begin{center}
 \begin{tabular}{c c}
 \AxiomC{$\{ \Gamma, \bar{\phi}_i \Rightarrow \bar{\psi}_i, \Delta \}_i$}
 \UnaryInfC{$\Gamma, \phi \Rightarrow \Delta$}
 \DisplayProof,
\end{tabular}
\end{center}
then there are two cases to consider, i.e., either $\phi \in \bar{C}$ or $\phi \in S^a$. If $\phi \in \bar{C}$, then set $\bar{C}'=\bar{C}-\{\phi\}$. Since $\phi \in \bar{C}$, it is $p^{\circ}$-free by the assumption and $\phi_i$'s are all $p^{\circ}$-free and $\psi_i$'s are all $p^{\diamond}$-free by the variable preserving property. By induction hypothesis, as $(\bar{C}', \bar{\phi}_i \Rightarrow \bar{\psi}_i, \bar{D})$ is $p^{\diamond}$-free and $S \cdot (\bar{C}', \bar{\phi}_i \Rightarrow \bar{\psi}_i, \bar{D})$ has a shorter proof, we have $\bar{C}', \bar{\phi}_i \Rightarrow \Ap S,  \bar{\psi}_i, \bar{D}$. By using the rule itself, we have  
\begin{center}
\begin{tabular}{c c c}
 \AxiomC{$ \{ \bar{C}', \bar{\phi}_i \Rightarrow \bar{\psi}_i, \Ap S, \bar{D} \}_i$} 
 \UnaryInfC{$\bar{C}', \phi \Rightarrow \Ap S, \bar{D}$}
 \DisplayProof
\end{tabular}
\end{center}
which implies $\bar{C} \Rightarrow \Ap S, \bar{D}$.\\
If $\phi \notin \bar{C}$, then both $\bar{C}$ and $\bar{D}$ 
do not 
contain any active formula of the rule and hence the last rule is in form
\begin{center}
\begin{tabular}{c c c}
 \AxiomC{$ \{ \bar{C}, \Gamma, \bar{\phi}_i \Rightarrow \bar{\psi}_i, \bar{D}, \Delta \}_i$} 
 \UnaryInfC{$\bar{C}, \Gamma, \phi \Rightarrow \bar{D}, \Delta$}
 \DisplayProof.
\end{tabular}
\end{center}
By context restriction property, if we erase $\bar{C}$ and $\bar{D}$ both on the premises and the consequence of the last rule, the rule remains valid and it changes to:
\begin{center}
\begin{tabular}{c c c}
 \AxiomC{$ \{ \Gamma, \bar{\phi}_i \Rightarrow \bar{\psi}_i, \Delta \}_i$} 
 \UnaryInfC{$ \Gamma, \phi \Rightarrow \Delta$}
 \DisplayProof.
\end{tabular}
\end{center}
Therefore, the rule is backward applicable to 
$S=(\Gamma, \phi \Rightarrow \Delta)$. Set $S_i=(\Gamma, \bar{\phi}_i \Rightarrow \bar{\psi}_i, \Delta)$. As the weight of $S_i$'s are less than the weight of $S$ and $S_i \cdot (\bar{C} \Rightarrow \bar{D})$ are provable, by induction hypothesis, we have $\bar{C} \Rightarrow \Ap S_i, \bar{D}$. Hence, $\bar{C} \Rightarrow \bigwedge_i \Ap S_i, \bar{D}$ and as $\bigwedge_i \Ap S_i$ is a disjunct in $\Ap S$, we have $\bar{C} \Rightarrow \Ap S, \bar{D}$.\\
The case where the last rule is in $\mathbf{GW3}$ with its main formula in the antecedent 
is similar. For the modal or conditional rules, by induction hypothesis $\mathcal{U}^{\circ}_p(S)$ 
and the assumption that $G$ has 
MULIP, 
we get that $\Apm S$ is a uniform $\A^{\circ}_p$-interpolant for $S$ with respect to the set of 
modal or conditional rules of $G$. By $(ii')$ in Definition \ref{DfnUniformInterpolationSeq}, this gives $\bar{C} \Rightarrow \Apm S, \bar{D}$ and hence $\bar{C} \Rightarrow \Ap S, \bar{D}$.
\end{proof}

\subsection{Uniform Lyndon Interpolation} \label{ULIPSubsection}

In this subsection, for the following choices of the system $G$, we show that it has MULIP. Therefore,  
by Theorem \ref{GeneralArgument} and Theorem \ref{SequentImpliesLogic}, we will have:
\begin{theorem} The logics $\mathsf{E}$, $\mathsf{M}$, $\mathsf{MC}$, $\mathsf{EN}$, $\mathsf{MN}$, $\mathsf{K}$, their conditional versions $\mathsf{CE}$, $\mathsf{CM}$, $\mathsf{CMC}$, $\mathsf{CEN}$, $\mathsf{CMN}$, $\mathsf{CK}$, and the conditional logic $\mathsf{CKID}$ have ULIP and hence UIP and LIP.
\end{theorem}

\subsubsection{Modal logics $\mathsf{M}$ and $\mathsf{MN}$}\label{Subsecionmmn}
Let $G$ be either $\mathbf{GM}$ or $\mathbf{GMN}$. We will show that $G$ has MULIP. 
To define $\Apm S$, if $\neg \, \mathcal{U}^{\circ}_p(S)$, define $\Apm S$ as $\bot$. If $ \mathcal{U}^{\circ}_p(S)$, (i.e., for any sequent $T$ lower than $S$ a 
uniform $\A^{\circ}_p$-interpolant, denoted by $\Ap T$, exists), define $\Apm S$ in the following way: if $S$ is provable, define 
it as $\top$, otherwise, if it is of the form $(\Box \phi \Rightarrow \,)$, define $\Apm S=\neg \Box \neg \Ap S'$, where $S'=(\phi \Rightarrow \,)$, if $S$ is of the form $(\, \Rightarrow \Box \psi)$, define $\Apm S= \Box \Ap S''$, where $S''=(\Rightarrow \psi)$, and otherwise, define $\Apm S=\bot$. Note that $\Apm S$ is well-defined as we have $ \mathcal{U}^{\circ}_p(S)$ and $S'$ 
and $S''$ 
are lower than $S$.\\ 
To show that $G$ has MULIP, 
we assume $ \mathcal{U}^{\circ}_{p}(S)$ to prove the three conditions $(var)$, $(i)$ and $(ii')$ in Definition \ref{DfnUniformInterpolationSeq} for $\Apm S$. First, note that using  $\mathcal{U}^{\circ}_p(S)$ on $(\phi \Rightarrow \,)$ and $(\, \Rightarrow \psi)$ that are lower than $(\Box \phi \Rightarrow \,)$ and $(\, \Rightarrow \Box \psi)$, respectively, the variable conditions are implied from $(var)$ for $S'$ and $S''$, respectively.\\
For $(i)$, if $S$ is provable, there is nothing to prove. 
Otherwise, if $S=(\Box \phi \Rightarrow \,)$ 
then $\Apm S=\neg \Box \neg \Ap S'$. 
As $S'$ is lower than $S$, we have $(\phi, \Ap S' \Rightarrow \,)$ by $\mathcal{U}^{\circ}_p(S)$, which implies $(\phi \Rightarrow \neg \Ap S')$. Using the rule $(M)$, we get $(\Box \phi \Rightarrow \Box \neg \Ap S')$, which is equivalent to $(\Box \phi, \neg \Box \neg \Ap S' \Rightarrow \,)$. Hence, $S \cdot  (\Apm S \Rightarrow \,)$ is provable.\\
If $S$ is not provable and $S=(\, \Rightarrow \Box \psi)$, 
we have $\Apm S=\Box \Ap S''$. Using $ \mathcal{U}^{\circ}_p(S)$ on $S''$ and the fact that $S''$ is lower than $S$, we have $(\Ap S'' \Rightarrow \psi)$ and by the rule $(M)$, we can show that $S \cdot (\Box \Ap S'' \Rightarrow \,)$ is provable in $G$. 
If $S$ is not provable and has none of the mentioned forms, as $\Apm S=\bot$, there is nothing to prove.\\
For $(ii')$, let $S \cdot (\bar{C} \Rightarrow \bar{D})$ be derivable in $G$ for a $p^{\diamond}$-free sequent $\bar{C} \Rightarrow \bar{D}$ and the last rule is a modal rule.
We want to show that $\bar{C} \Rightarrow \Apm S, \bar{D}$ is derivable in $G$. If the last rule used in the proof of $S \cdot (\bar{C} \Rightarrow \bar{D})$ is $(M)$, the sequent must have the form $(\Box \phi \Rightarrow \Box \psi)$ and the rule must be in form:
\begin{center}
\begin{tabular}{c}
 \AxiomC{$\phi \Rightarrow \psi$} 
   \RightLabel{\footnotesize$M$} 
 \UnaryInfC{$\Box \phi \Rightarrow \Box \psi$}
 \DisplayProof
\end{tabular}
\end{center}
If $S$ is provable, as $\Apm S=\top$, we clearly have $\bar{C} \Rightarrow \Apm S, \bar{D}$. 
Assume $S$ is not provable and hence $\bar{C} \cup \bar{D}$ cannot be empty. Therefore, there are three cases to consider, either $\bar{C}$ is $\Box \phi$ or $\bar{D}$ is $\Box \psi$ or both. 
First, if $\bar{C}=\Box \phi $ and $\bar{D}=\varnothing$, then, 
$S=(\, \Rightarrow \Box \psi)$ and $\phi$ is $p^{\circ}$-free. Set $S''=(\, \Rightarrow \psi)$. 
Then $\Apm S=\Box \forall p S''$. As $S''$ is lower than $S$, by $\mathcal{U}^{\circ}_p(S)$ we have $(\phi \Rightarrow \Ap S'')$. Using the modal rule $(M)$, we have $(\Box \phi \Rightarrow \Box \Ap S'')$ and hence $(\bar{C} \Rightarrow \Apm S, \bar{D})$.\\
In the second case, assume $\bar{C}=\varnothing$ and $\bar{D}=\Box \psi$. Hence, $S=(\Box \phi \Rightarrow \,)$ and $\psi$ is $p^{\diamond}$-free. Set $S'=(\phi \Rightarrow \,)$. 
Hence, $\Apm S=\neg \Box \neg \Ap S'$. Since $(\phi \Rightarrow \psi)$ is provable in $G$ and $S'$ is lower than $S$, by $\mathcal{U}^{\circ}_p$ we have $(\, \Rightarrow \Ap S', \psi)$, or equivalently $(\neg \Ap S' \Rightarrow \psi)$. Using the 
rule $(M)$, we get $(\Box \neg \Ap S' \Rightarrow \Box \psi)$ or equivalently $(\Rightarrow \neg \Box \neg \Ap S', \Box \psi)$. Therefore, we have $(\, \Rightarrow \Apm S, \Box \psi)$ or $(\bar{C} \Rightarrow \Apm S, \bar{D})$. \\
In the third case, if $\bar{C}=\Box \phi $ and $\bar{D}=\Box \psi$, then $S$ is the empty sequent and $\bar{C} \Rightarrow \bar{D}$ is provable. Hence, $\bar{C} \Rightarrow \Apm S, \bar{D}$ is also provable.\\
For the case $G=\mathbf{GMN}$, if $S \cdot (\bar{C} \Rightarrow \bar{D})=(\Rightarrow \Box \psi)$ is proved by the rule $(N)$, it must have the following form:
\begin{center}
\begin{tabular}{c}
 \AxiomC{$ \Rightarrow \psi$} 
   \RightLabel{\footnotesize $N$} 
 \UnaryInfC{$ \Rightarrow \Box \psi$}
 \DisplayProof
\end{tabular}
\end{center}
Then $\bar{C}= \varnothing$ and there are two cases to consider. The first case is when $S=(\Rightarrow \Box \psi)$ and $\bar{D}=\varnothing$. Then, it means that $S$ is provable which contradicts our assumption. The second case is when $S= (\, \Rightarrow )$ and $\bar{D}= \Box \psi$. Hence, $\bar{C} \Rightarrow \bar{D}$ is provable and we have the provability of $\bar{C} \Rightarrow \Apm S, \bar{D}$ in $G$.

\subsubsection{Modal logics $\mathsf{MC}$ and $\mathsf{K}$}
Let $G$ be either $\mathbf{GMC}$ or $\mathbf{GK}$. We will show that $G$ has MULIP. Similar to the argument of the previous subsection, to define $\Apm S$, if  $\neg \, \mathcal{U}^{\circ}_p(S)$, define $\Apm S$ as $\bot$. If $\mathcal{U}^{\circ}_p(S) $, (i.e., for any sequent $T$ lower than $S$ the uniform $\A^{\circ}_p$-interpolant, denoted by $\Ap T$, exists), define $\Apm S$ as the following: if $S$ is provable, define $\Apm S=\top$. Otherwise, if $S$ is of the form $(\Box \phi_1, \cdots, \Box \phi_i \Rightarrow \,)$, for some $i \geq 1$, define $\Apm S=\neg \Box \neg \Ap S'$, where $S'=(\phi_1, \cdots, \phi_i \Rightarrow \,)$. If $S$ is of the form $(\, \Rightarrow \Box \psi)$, define $\Apm S=\Box \Ap S''$, where $S''=(\, \Rightarrow \psi)$. If $S$ is of the form $(\Box \phi_1, \cdots, \Box \phi_i \Rightarrow \Box \psi)$, for some $i \geq 1$, define $\Apm S=\Box \Ap S''$, where $S''=(\phi_1, \cdots, \phi_i \Rightarrow \psi)$. Otherwise, define $\Apm S=\bot$. Note that $\Apm S$ is well-defined as we assumed $\mathcal{U}^{\circ}_p(S)$ and in each case $S'$ or $S''$ are lower than $S$.\\
To show that $G$ has MULIP, 
we assume $\mathcal{U}^{\circ}_p(S)$ to prove the three conditions $(var)$, $(i)$ and $(ii')$ in Definition \ref{DfnUniformInterpolationSeq} for $\Apm S$. The condition $(var)$ is an immediate consequence of $\mathcal{U}^{\circ}_p(S)$ and the fact that $S'$ or $S''$ are lower than $S$. For $(i)$, if $S$ is provable, there is nothing to prove. If $S$ is of the form $(\Box \phi_1, \cdots, \Box \phi_i \Rightarrow \,)$ and $\Apm S=\neg \Box \neg \Ap S'$, where $S'=(\phi_1, \cdots, \phi_i \Rightarrow \,)$, as $S'$ is lower than $S$, by $\mathcal{U}^{\circ}_p(S)$ we have $(\phi_1, \cdots, \phi_i, \Ap S' \Rightarrow)$ or equivalently $(\phi_1, \cdots, \phi_i \Rightarrow \neg \Ap S')$. Using the rule $(MC)$, we get $(\Box \phi_1, \cdots, \Box \phi_i \Rightarrow \Box \neg \Ap S')$, which is equivalent to $(\Box \phi_1, \cdots, \Box \phi_i , \neg \Box \neg \Ap S' \Rightarrow \,)$ and hence $S \cdot (\Apm S \Rightarrow \,)$.\\
If $S$ is of the form $(\, \Rightarrow \Box \psi)$ and $S''=(\, \Rightarrow \psi)$, or $S$ is of the form $S=(\Box \phi_1, \cdots, \Box \phi_i \Rightarrow \Box \psi)$, for some $i \geq 1$ and $S''$ is of the form $(\phi_1, \cdots, \phi_i \Rightarrow \psi)$, we have $\Apm S=\Box \Ap S''$. In both cases, using $\mathcal{U}^{\circ}_p(S)$ on $S''$, we have either $\Ap S'' \Rightarrow \psi$ or $\phi_1, \cdots, \phi_i, \Ap S'' \Rightarrow \psi$, respectively. In both cases, using the rule $(MC)$, we can show that $S \cdot (\Box \Ap S'' \Rightarrow \,)$ is provable and hence $S \cdot (\Apm S \Rightarrow \,)$.\\
For $(ii')$, let $S \cdot (\bar{C} \Rightarrow \bar{D})$ be derivable in $G$ and the last rule is the modal rule $(MC)$, for a $p^{\diamond}$-free sequent $\bar{C} \Rightarrow \bar{D}$. We want to show that $\bar{C} \Rightarrow \Apm S, \bar{D}$ is derivable in $G$. If $S$ is provable, as $\Apm S=\top$, we have $\bar{C} \Rightarrow \Apm S, \bar{D}$. Therefore, we assume that $S$ is not provable.
As the last rule used in the proof of $S \cdot (\bar{C} \Rightarrow \bar{D})$ is $(MC)$, the sequent must have the form $(\Box \phi_1 , \cdots , \Box \phi_n \Rightarrow \Box \psi)$ and the rule is:
\begin{center}
\begin{tabular}{c}
 \AxiomC{$\phi_1, \cdots, \phi_n \Rightarrow \psi$} 
   \RightLabel{\footnotesize$MC$} 
 \UnaryInfC{$\Box \phi_1, \cdots, \Box \phi_n \Rightarrow \Box \psi$}
 \DisplayProof
\end{tabular}
\end{center}
Then, there are two cases to consider, either $\bar{D}=\Box \psi$ or $\bar{D}=\varnothing$. First, assume $S$ is of the form $(\Box \phi_1 , \cdots , \Box \phi_i \Rightarrow \,)$, for $i \leq n$, then $\bar{C}=\Box \phi_{i+1}, \cdots , \Box \phi_n$ and $\bar{D}=\Box \psi$ and hence $\phi_{i+1}, \cdots, \phi_n \Rightarrow \psi$ is $p^{\diamond}$-free. Set $S'=(\phi_1, \cdots , \phi_i \Rightarrow \,)$. By the form of $S$, we have $\Apm S=\neg \Box \neg \Ap S'$. As $S'$ is lower than $S$, by $\mathcal{U}^{\circ}_p(S)$, we have $(\phi_{i+1}, \cdots , \phi_n \Rightarrow \Ap S', \psi)$. Hence, by moving $\Ap S'$ to the left, applying the rule $(MC)$ and moving back, we have $(\Box \phi_{i+1}, \cdots , \Box \phi_n \Rightarrow \neg \Box \neg \Ap S', \Box \psi)$ or equivalently $(\bar{C} \Rightarrow \Apm S, \bar{D})$.\\
If $S$ is of the form $\Box \phi_1, \cdots, \Box \phi_i \Rightarrow \Box \psi$, for some $i \leq n$, we must have $\bar{C}= \Box \phi_{i+1}, \cdots, \Box \phi_n$ and $\bar{D}=\varnothing$. Hence, $\phi_{i+1}, \cdots, \phi_n$ are $p^{\circ}$-free. Note that $i < n$, because if $i=n$, then $S$ will be provable that contradicts our assumption. Set $S''=(\phi_1, \cdots, \phi_i \Rightarrow \psi)$. As $S''$ is lower than $S$, by $\mathcal{U}^{\circ}_p(S)$ we have $ \phi_{i+1}, \cdots, \phi_n \Rightarrow \Ap S''$. By the fact that $i < n$, we can apply the rule $(MC)$ to prove $\Box \phi_{i+1}, \cdots, \Box \phi_n \Rightarrow \Box \Ap S''$ and hence $(\bar{C} \Rightarrow \Apm S, \bar{D})$.\\

\noindent The case in which $G=\mathbf{GMK}$ and the last rule is $(N)$ is similar to the case $G = \mathbf{GMN}$ where the last rule is $(N)$, as discussed in Subsubsection \ref{Subsecionmmn}.

\subsubsection{Modal logics $\mathsf{E}$ and $\mathsf{EN}$}

Let $G$ be $\mathbf{GE}$ or $\mathbf{GEN}$. Similar to the argument of the previous subsection, 
if  $\neg \, \mathcal{U}^{\circ}_p(S)$, define $\Apm S$ as $\bot$. If $\mathcal{U}^{\circ}_p(S) $, 
then: if $S$ is provable in $G$, define $\Apm S=\top$. Otherwise, if 
$S=(\Box \phi \Rightarrow \,)$ and 
both $(\neg \Ap S' \Rightarrow \phi)$ and $(\phi \Rightarrow \neg \Ap S')$ are provable in $G$, 
define $\Apm S=\neg \Box \neg \Ap S'$ for $S'= (\phi \Rightarrow \,)$. If $S$ has the form $(\, \Rightarrow \Box \psi)$ and 
both $(\Ap S'' \Rightarrow \psi)$ and $(\psi \Rightarrow \Ap S'')$ are provable in $G$, 
define $\Apm S=\Box \Ap S''$ for $S''= (\, \Rightarrow \psi)$. Otherwise, define $\Apm S=\bot$. Note that $\Ap S$ is well-defined as 
$S'$ and $S''$ are lower than $S$ and we assumed $\mathcal{U}^{\circ}_p(S)$.\\
To show that $G$ has MULIP 
we assume $\mathcal{U}^{\circ}_p(S)$ to prove 
$(var)$, $(i)$ and $(ii')$ in Definition \ref{DfnUniformInterpolationSeq} for $\Apm S$. 
Condition $(var)$ is a 
consequence of $\mathcal{U}^{\circ}_p(S)$ and 
that $S'$ or $S''$ are lower than $S$. For $(i)$, if $S$ is provable, there is nothing to prove. If $S=(\Box \phi \Rightarrow \,)$ and $S'=(\phi \Rightarrow \,)$ and 
both $(\neg \Ap S' \Rightarrow \phi)$ and $(\phi \Rightarrow \neg \Ap S')$ are provable in $G$, then using the rule $(E)$, we have $(\Box \phi \Rightarrow \Box \neg \Ap S')$ which implies $(\Box \phi, \neg \Box \neg \Ap S' \Rightarrow \,)$ and hence $S \cdot (\Apm S \Rightarrow \, )$ is provable in $G$.\\
If $S=(\, \Rightarrow \Box \psi)$ and $S''=(\, \Rightarrow \psi)$ and 
both $( \Ap S'' \Rightarrow \psi)$ and $(\psi \Rightarrow  \Ap S'')$ are provable in $G$, then using the rule $(E)$, we have $(\Box \Ap S'' \Rightarrow \Box \psi)$ and hence $S \cdot (\Apm S \Rightarrow \, ) $ is provable in $G$. If $\Apm S=\bot$, there is nothing to prove.\\
For $(ii')$, if $S$ is provable, then $\Apm S=\top$ and hence $\bar{C} \Rightarrow \Apm S, \bar{D}$. 
Therefore, assume that $S$ is not provable. If the last rule used in the proof of $ S \cdot (\bar{C} \Rightarrow \bar{D})$ is the rule $(E)$, the sequent $ S \cdot (\bar{C} \Rightarrow \bar{D})$ is of the form $\Box \phi \Rightarrow \Box \psi$. There are four cases to consider based on if $\bar{C}$ or $\bar{D}$ are empty or not. First, if $\bar{C}=\bar{D}=\varnothing$, then $S$ is provable which contradicts our assumption. If $S$ is the empty sequent $(\, \Rightarrow \,)$, then $\bar{C} \Rightarrow \bar{D}$ is provable and hence $\bar{C} \Rightarrow \Apm S, \bar{D}$ is provable.\\
If 
$S=(\Box \phi \Rightarrow \,)$, then $\bar{C}=\varnothing$ and $\bar{D}=\Box \psi$ and hence $\psi$ is $p^{\diamond}$-free. Set $S'=(\phi \Rightarrow \,)$ and 
as the last rule is $(E)$, both 
$\phi \Rightarrow \psi$ and $\psi \Rightarrow \phi$ are provable. By $\mathcal{U}^{\circ}_p(S)$ and the fact that $S'$ is lower than $S$, we have $(\phi , \Ap S' \Rightarrow )$ or equivalently, $(\phi \Rightarrow \neg \Ap S')$. Again by $\mathcal{U}^{\circ}_p(S)$ for $S'$, the provability of $S' \cdot (\, \Rightarrow \bar{D})=(\phi \Rightarrow \psi)$ and the fact that $(\, \Rightarrow \psi)$ is $p^{\diamond}$-free, we have $(\, \Rightarrow \Ap S', \psi)$ or equivalently, $(\neg \Ap S' \Rightarrow \psi)$. Since $(\phi \Rightarrow \psi)$ and $(\psi \Rightarrow \phi)$ are provable, by cut we can prove the equivalence between $\phi$, $\psi$ and $\neg \Ap S'$. Using this fact, we have:
\begin{center}
\begin{tabular}{c c}
 \AxiomC{$\psi \Rightarrow \neg \Ap S'$} 
  \AxiomC{$ \neg \Ap S' \Rightarrow \psi$} 
   \RightLabel{\footnotesize$E$} 
 \BinaryInfC{$ \Box  \neg \Ap S' \Rightarrow \Box \psi$}
 \DisplayProof
\end{tabular}
\end{center}
Hence, $( \Rightarrow \neg \Box \neg \Ap S' , \Box \psi)$. Then, as 
$S=(\Box \phi \Rightarrow \, )$ and both 
$(\neg \Ap S' \Rightarrow \phi)$ and $(\phi \Rightarrow \neg \Ap S')$ are provable in $G$, by definition we have $\Apm S=\neg \Box  \neg \Ap S'$ and hence $(\, \Rightarrow  \neg \Box  \neg \Ap S', \Box \psi) = (\bar{C} \Rightarrow \Apm S, \bar{D})$ is provable in $G$. 
The last case where 
$S=(\, \Rightarrow \Box \psi)$ and $\bar{C}=\Box \phi$ and $\bar{D}=\varnothing$ is similar.\\
For the case $G=\mathbf{GEN}$, if $S \cdot (\bar{C} \Rightarrow \bar{D})=(\, \Rightarrow \Box \psi)$ is proved by the rule $(N)$, it must have the 
form
\begin{tabular}{c}
 \AxiomC{$ \Rightarrow \psi$} 
   \RightLabel{\footnotesize$N$} 
 \UnaryInfC{$ \Rightarrow \Box \psi$}
 \DisplayProof.
\end{tabular}
Then $\bar{C}= \varnothing$ and there are two cases. 
First, $S=(\Rightarrow \Box \psi)$ and $\bar{D}=\varnothing$, which means 
that $S$ is provable which contradicts our assumption. 
Second, if $S= (\, \Rightarrow )$ and $\bar{D}= \Box \psi$, 
and hence $\bar{C} \Rightarrow \bar{D}$ is provable, we have the provability of $\bar{C} \Rightarrow \Apm S, \bar{D}$ in $G$.

\subsubsection{Modal logics $\mathsf{CE}$ and $\mathsf{CEN}$} \label{CE}
Let $G$ be either $\mathbf{GCE}$ or $\mathbf{GCEN}$. We will show that $G$ has MULIP. To define $\Apm S$, if  $\neg \, \mathcal{U}^{\circ}_p(S)$, define $\Apm S$ as $\bot$. If $\mathcal{U}^{\circ}_p(S) $, define $\Apm S$ as the following: if $S$ is provable, define $\Apm S=\top$. Otherwise, if $S$ is of the form $S=(\phi_1 \triangleright \psi_1 \Rightarrow \,)$, define $T'=(\, \Rightarrow \phi_1)$ and $S'=(\psi_1 \Rightarrow \,)$. Then, if $G \vdash \Ap T' \Leftrightarrow \phi_1$ and $G \vdash \neg \Ap S' \Leftrightarrow \psi_1$, define $\Apm S=\neg (\Ap T' \triangleright  \neg \Ap S')$. If $S$ is of the form $S=(\, \Rightarrow \phi_0 \triangleright \psi_0)$, define
$T''=(\phi_0 \Rightarrow \,)$ and $S''=(\, \Rightarrow \psi_0)$. Then, if $G \vdash \neg \Ap T'' \Leftrightarrow \phi_0$ and $G \vdash \Ap S'' \Leftrightarrow \psi_0$, define $\Apm S=\neg \Ap T'' \triangleright \Ap S''$.  Otherwise, define $\Apm S=\bot$. Note that $\Apm S$ is well-defined as in each case $S'$, $T'$, $S''$ and $T''$ are lower than $S$, and we assumed $\mathcal{U}^{\circ}_p(S)$.\\
To show that $\mathbf{GCE}$ has MULIP, we assume $\mathcal{U}^{\circ}_p(S)$ to prove the three conditions $(var)$, $(i)$ and $(ii')$ in Definition \ref{DfnUniformInterpolationSeq} for $\Apm S$,  a uniform $\A^{\circ}_p$-interpolant for $S$ with respect to the set $\mathcal{M}$ of the conditional rules of $G$. The condition $(var)$ is an immediate consequence of the fact that $\mathcal{U}^{\circ}_p(S)$ holds and $S'$, $T'$, $S''$ and $T''$ are lower than $S$.\\ 
For $(i)$, if $S$ is provable in $G$, there is nothing to prove. Hence, assume that $S$ is not provable. If $S$ is of the form $S=(\phi_1 \triangleright \psi_1 \Rightarrow \,)$ and we have $G \vdash \Ap T' \Leftrightarrow \phi_1$ and $G \vdash \neg \Ap S' \Leftrightarrow \psi_1$, where $T'=(\, \Rightarrow \phi_1)$ and $S'=(\psi_1 \Rightarrow \,)$, then by definition, we have $\Apm S=\neg (\Ap T' \triangleright \neg \Ap S')$. As $G \vdash \Ap T' \Leftrightarrow \phi_1$ and $G \vdash \neg \Ap S' \Leftrightarrow \psi_1$, using the rule $(CE)$, we get $G \vdash (\phi_1 \triangleright \psi_1 \Rightarrow \Ap T' \triangleright \neg \Ap S')$, which is equivalent to $G \vdash (\phi_1 \triangleright \psi_1, \neg (\Ap T' \triangleright \neg \Ap S') \Rightarrow \,)$ and hence $G \vdash S \cdot (\Apm S \Rightarrow \,)$.\\
If $S$ is of the form $S=(\; \Rightarrow \phi_0 \triangleright \psi_0)$ and we have
$G \vdash \neg \Ap T'' \Leftrightarrow \phi_0$ and $G \vdash \Ap S'' \Leftrightarrow \psi_0$, where $T''=(\phi_0 \Rightarrow \,)$ and $S''=(\, \Rightarrow \psi_0)$, then by definition $\Apm S=\neg \Ap T'' \triangleright \Ap S''$. As $G \vdash \neg \Ap T'' \Leftrightarrow \phi_0$ and $G \vdash \Ap S'' \Leftrightarrow \psi_0$, by the rule $(CE)$, we can show that $G \vdash \neg \Ap T'' \triangleright \Ap S'' \Rightarrow \phi_0 \triangleright \psi_0$ and hence, $G \vdash S \cdot (\Apm S \Rightarrow \,)$.\\
For $(ii')$, let $S \cdot (\bar{C} \Rightarrow \bar{D})$ be derivable in $G$ and the last rule is $(CE)$, for a $p^{\diamond}$-free sequent $\bar{C} \Rightarrow \bar{D}$. We want to show that $\bar{C} \Rightarrow \Apm S, \bar{D}$ is derivable in $G$. If $S$ is provable, as $\Apm S=\top$, we have $\bar{C} \Rightarrow \Apm S, \bar{D}$. Therefore, we assume that $S$ is not provable.
As the last rule used in the proof of $S \cdot (\bar{C} \Rightarrow \bar{D})$ is $(CE)$, the sequent must have the form $(\phi_1 \triangleright \psi_1 \Rightarrow \phi_0 \triangleright \psi_0)$ and the rule is:
\begin{center}
\begin{tabular}{c}
\AxiomC{$\{\phi_0 \Rightarrow \phi_1 \;\;\; , \;\;\; \phi_1 \Rightarrow \phi_0\}$} 
 \AxiomC{$\{\psi_0 \Rightarrow \psi_1 \;\;\; , \;\;\; \psi_1 \Rightarrow \psi_0\}$} 
   \RightLabel{\footnotesize$CE$} 
 \BinaryInfC{$\phi_1 \triangleright \psi_1 \Rightarrow \phi_0 \triangleright \psi_0$}
 \DisplayProof
\end{tabular}
\end{center}
Then, there are two cases to consider, either $\bar{D}=\{\phi_0 \triangleright \psi_0\}$ or $\bar{D}=\varnothing$. First, assume that $\bar{D}=\{\phi_0 \triangleright \psi_0\}$. Then, if $\bar{C}=\{\phi_1 \triangleright \psi_1\}$, the sequent $(\bar{C} \Rightarrow \Apm S, \bar{D})$ is clearly provable. Hence, we assume that $\bar{C}=\varnothing$ which implies $S=(\phi_1 \triangleright \psi_1 \Rightarrow \,)$. As $\bar{D}=\{\phi_0 \triangleright \psi_0\}$, the formulas $\phi_0$ and $\psi_0$ are $p^{\circ}$-free and $p^{\diamond}$-free, respectively. Set $T'=(\, \Rightarrow \phi_1)$ and $S'=(\psi_1 \Rightarrow \,)$. By $\mathcal{U}^{\circ}_p(S)$ and the fact that $T'$ is lower than $S$, we know that $\forall^{\circ}p T'$ exists, $G \vdash \Ap T' \Rightarrow \phi_1$ and since $\phi_0$ is $p^{\circ}$-free and $G \vdash \phi_0 \Rightarrow \phi_1$, we have $G \vdash \phi_0 \Rightarrow \Ap T'$. As $G \vdash \phi_0 \Leftrightarrow \phi_1$, we have $G \vdash \phi_1 \Rightarrow \Ap T'$  and hence $G \vdash \Ap T' \Leftrightarrow \phi_1$. Again, by $\mathcal{U}^{\circ}_p(S)$ and the fact that $S'$ is lower than $S$, we know that $\forall^{\circ}p S'$ exists and $G \vdash (\Ap S', \psi_1 \Rightarrow \,)$ which implies $G \vdash \psi_1 \Rightarrow \neg \Ap S'$. Since $\psi_0$ is $p^{\diamond}$-free and $G \vdash \psi_1 \Rightarrow \psi_0$, we have $G \vdash (\, \Rightarrow \Ap S', \psi_0)$ which implies $G \vdash \neg \Ap S' \Rightarrow \psi_0$. As $G \vdash \psi_0 \Leftrightarrow \psi_1$, we have $G \vdash   \neg \Ap S' \Rightarrow \psi_1$  and hence $G \vdash \neg \Ap S' \Leftrightarrow \psi_1$. 
Now, by definition, we have $\Apm S=\neg (\Ap T' \triangleright \neg \Ap S')$. We have shown than $G \vdash \Ap T' \Leftrightarrow \phi_1$. As $G \vdash \phi_0 \Leftrightarrow \phi_1$, we have $G \vdash \Ap T' \Leftrightarrow \phi_0$. Similarly, we have shown that $G \vdash \neg \Ap S' \Leftrightarrow \psi_1$. As $G \vdash \psi_0 \Leftrightarrow \psi_1$, we have $G \vdash \neg \Ap S' \Leftrightarrow \psi_0$. Hence, by applying the rule $(CE)$, we have $G \vdash \Ap T' \triangleright \neg \Ap S' \Rightarrow \phi_0 \triangleright \psi_0$ or equivalently $G \vdash (\bar{C} \Rightarrow \Apm S, \bar{D})$.\\
For the second case, let $\bar{D}=\varnothing$. If $\bar{C}=\varnothing$, the sequent $S$ is provable which is a contradiction. Hence, $\bar{C}=\{\phi_1 \triangleright \psi_1\}$ which implies that $S=(\; \Rightarrow \phi_0 \triangleright \psi_0)$. 
Since $\bar{C}=\{\phi_1 \triangleright \psi_1\}$, the formulas $\phi_1$ and $\psi_1$ are $p^{\diamond}$-free and $p^{\circ}$-free, respectively. Set $T''=(\phi_0 \Rightarrow \,)$ and $S''=(\, \Rightarrow \psi_0)$. Then, by $\mathcal{U}^{\circ}_p(S)$ and the fact that $T''$ is lower than $S$, we know that $\forall^{\circ}p T''$ exists and $G \vdash (\phi_0, \Ap T'' \Rightarrow \,)$ which implies $G \vdash \phi_0 \Rightarrow \neg \Ap T''$. As $G \vdash \phi_0 \Rightarrow \phi_{1}$ and $\phi_{1}$ is $p^{\diamond}$-free, we have $G \vdash (\, \Rightarrow \Ap T'', \phi_{1})$ which implies $G \vdash \neg \Ap T'' \Rightarrow \phi_1$. As $G \vdash \phi_0 \Leftrightarrow \phi_1$, we have $G \vdash \neg \Ap T'' \Leftrightarrow \phi_0$. Again, by $\mathcal{U}^{\circ}_p(S)$ and the fact that $S''$ is lower than $S$, we know that $\forall^{\circ}p S''$ exists and $G \vdash (\, \Ap T'' \Rightarrow \psi_0)$. As $G \vdash \psi_1 \Rightarrow \psi_{0}$ and $\psi_{1}$ is $p^{\circ}$-free, we have $G \vdash (\psi_1 \Rightarrow \Ap S'')$. As $G \vdash \psi_0 \Leftrightarrow \psi_1$, we have $G \vdash \Ap S'' \Leftrightarrow \psi_0$.
Now, by definition, we have $\Apm S=\neg \Ap T'' \triangleright \Ap S''$. 
We have shown that $G \vdash \neg \Ap T'' \Leftrightarrow \phi_0$. As $G \vdash \phi_0 \Leftrightarrow \phi_1$, we have $G \vdash \neg \Ap T'' \Leftrightarrow \phi_1$. Similarly, we have shown that $G \vdash \Ap S'' \Leftrightarrow \psi_0$. As $G \vdash \psi_0 \Leftrightarrow \psi_1$, we have $G \vdash \Ap S'' \Leftrightarrow \psi_1$. Hence, by applying the rule $(CE)$, we have $G \vdash \phi_1 \triangleright \psi_1 \Rightarrow \neg \Ap T'' \triangleright \Ap S''$ or equivalently $G \vdash (\bar{C} \Rightarrow \Apm S, \bar{D})$. This concludes the case $G=\mathbf{GCE}$.\\

For the case $G=\mathbf{GCEN}$, if $S \cdot (\bar{C} \Rightarrow \bar{D})=(\; \Rightarrow \phi_0 \triangleright \psi_0)$ is proved by the rule $(CN)$, it must have the following form:
\begin{center}
\begin{tabular}{c}
 \AxiomC{$ \Rightarrow \psi_0$} 
   \RightLabel{\footnotesize $CN$} 
 \UnaryInfC{$ \Rightarrow \phi_0 \triangleright \psi_0$}
 \DisplayProof
\end{tabular}
\end{center}
Then $\bar{C}= \varnothing$ and there are two cases to consider. The first case is when $S=(\Rightarrow \phi_0 \triangleright \psi_0)$ and $\bar{D}=\varnothing$. Then, it means that $S$ is provable which contradicts our assumption. The second case is when $S= (\, \Rightarrow )$ and $\bar{D}= \phi_0 \triangleright \psi_0$. Hence, $\bar{C} \Rightarrow \bar{D}$ is provable and we have the provability of $\bar{C} \Rightarrow \Apm S, \bar{D}$ in $G$, by the weakening rule.

\subsubsection{Conditional logics $\mathsf{CM}$ and $\mathsf{CMN}$}
Let $G$ be either $\mathbf{GCM}$ or $\mathbf{GCMN}$. We will show that $G$ has MULIP. To define $\Apm S$, if  $\neg \, \mathcal{U}^{\circ}_p(S)$, define $\Apm S$ as $\bot$. If $\mathcal{U}^{\circ}_p(S) $, define $\Apm S$ as the following: if $S$ is provable, define $\Apm S=\top$. Otherwise, if $S$ is of the form $S=(\phi_1 \triangleright \psi_1 \Rightarrow \,)$, define $T'=(\, \Rightarrow \phi_1)$ and $S'=(\psi_1 \Rightarrow \,)$. Then, if $G \vdash \Ap T' \Leftrightarrow \phi_1$, define $\Apm S=\neg (\Ap T' \triangleright  \neg \Ap S')$. If $S$ is of the form $S=(\, \Rightarrow \phi_0 \triangleright \psi_0)$, define
$T''=(\phi_0 \Rightarrow \,)$ and $S''=(\, \Rightarrow \psi_0)$. Then, if $G \vdash \neg \Ap T'' \Leftrightarrow \phi_0$, define $\Apm S=\neg \Ap T'' \triangleright \Ap S''$.  Otherwise, define $\Apm S=\bot$. Note that $\Apm S$ is well-defined as in each case $S'$, $T'$, $S''$ and $T''$ are lower than $S$ and we assumed $\mathcal{U}^{\circ}_p(S)$.\\
To show that $G$ has MULIP, we assume $\mathcal{U}^{\circ}_p(S)$ to prove the three conditions $(var)$, $(i)$ and $(ii')$ in Definition \ref{DfnUniformInterpolationSeq} for $\Apm S$. The condition $(var)$ is an immediate consequence of the fact that $\mathcal{U}^{\circ}_p(S)$ holds and $S'$, $T'$, $S''$ and $T''$ are lower than $S$.\\
For $(i)$, if $S$ is provable in $G$, there is nothing to prove. Hence, assume that $S$ is not provable. If $S$ is of the form $S=(\phi_1 \triangleright \psi_1 \Rightarrow \,)$ and $G \vdash \Ap T' \Leftrightarrow \phi_1$, for $T'=(\, \Rightarrow \phi_1)$, then by definition, we have $\Apm S=\neg (\Ap T' \triangleright \neg \Ap S')$, where $S'=(\psi_1 \Rightarrow \,)$. As $S'$ is lower than $S$, by $\mathcal{U}^{\circ}_p(S)$, we have $G \vdash (\psi_1, \Ap S' \Rightarrow \;)$ or equivalently $G \vdash (\psi_1 \Rightarrow \neg \Ap S')$. As $G \vdash \Ap T' \Leftrightarrow \phi_1$, using the rule $(CM)$, we get $G \vdash (\phi_1 \triangleright \psi_1 \Rightarrow \Ap T' \triangleright \neg \Ap S')$, which is equivalent to $G \vdash (\phi_1 \triangleright \psi_1, \neg (\Ap T' \triangleright \neg \Ap S') \Rightarrow \,)$ and hence $G \vdash S \cdot (\Apm S \Rightarrow \,)$.\\
If $S$ is of the form $S=(\; \Rightarrow \phi_0 \triangleright \psi_0)$ and
$G \vdash \neg \Ap T'' \Leftrightarrow \phi_0$, for $T''=(\phi_0 \Rightarrow \,)$, then by definition $\Apm S=\neg \Ap T'' \triangleright \Ap S''$, where $S''=(\, \Rightarrow \psi_0)$. Using $\mathcal{U}^{\circ}_p(S)$ on $S''$, we have  $G \vdash \Ap S'' \Rightarrow \psi_0$ and since $G \vdash \neg \Ap T'' \Leftrightarrow \phi_0$, using the rule $(CM)$, we can show that $G \vdash \neg \Ap T'' \triangleright \Ap S'' \Rightarrow \phi_0 \triangleright \psi_0$ and hence $G \vdash S \cdot (\Apm S \Rightarrow \,)$.\\
For $(ii')$, let $S \cdot (\bar{C} \Rightarrow \bar{D})$ be derivable in $G$ and the last rule is $(CM)$, for a $p^{\diamond}$-free sequent $\bar{C} \Rightarrow \bar{D}$. We want to show that $G \vdash \bar{C} \Rightarrow \Apm S, \bar{D}$. If $S$ is provable, as $\Apm S=\top$, we have $\bar{C} \Rightarrow \Apm S, \bar{D}$. Therefore, we assume that $S$ is not provable.
As the last rule used in the proof of $S \cdot (\bar{C} \Rightarrow \bar{D})$ is $(CM)$, the sequent must have the form $(\phi_1 \triangleright \psi_1 \Rightarrow \phi_0 \triangleright \psi_0)$ and the rule is in the form:
\begin{center}
\begin{tabular}{c}
\AxiomC{$\{\phi_0 \Rightarrow \phi_1 \;\;\; , \;\;\; \phi_1 \Rightarrow \phi_0\}$} 
 \AxiomC{$\psi_1 \Rightarrow \psi_0$} 
   \RightLabel{\footnotesize $CM$} 
 \BinaryInfC{$\phi_1 \triangleright \psi_1 \Rightarrow \phi_0 \triangleright \psi_0$}
 \DisplayProof
\end{tabular}
\end{center}
Then, there are two cases to consider, either $\bar{D}=\{\phi_0 \triangleright \psi_0\}$ or $\bar{D}=\varnothing$. First, assume that $\bar{D}=\{\phi_0 \triangleright \psi_0\}$. Then, if $\bar{C}=\{\phi_1 \triangleright
\psi_1\}$, the sequent $(\bar{C} \Rightarrow \Apm S, \bar{D})$ is clearly provable. Hence, we assume that $\bar{C}=\varnothing$ which implies $S=(\phi_1 \triangleright \psi_1 \Rightarrow \,)$. As $\bar{D}=\{\phi_0 \triangleright \psi_0\}$ and $\bar{D}$ is $p^{\diamond}$-free, we know that $\phi_0$ and $\psi_0$ are $p^{\circ}$-free and $p^{\diamond}$-free, respectively. Set $T'=(\, \Rightarrow \phi_1)$ and $S'=(\psi_1 \Rightarrow \,)$. By $\mathcal{U}^{\circ}_p(S)$ and the fact that $T'$ is lower than $S$, we know that $\forall^{\circ}p T'$ exists, $G \vdash \Ap T' \Rightarrow \phi_1$, and since $\phi_0$ is $p^{\circ}$-free and $G \vdash \phi_0 \Rightarrow \phi_1$, we have $G \vdash \phi_0 \Rightarrow \Ap T'$. As $G \vdash \phi_0 \Leftrightarrow \phi_1$, we have $G \vdash \phi_1 \Rightarrow \Ap T'$  and hence $G \vdash \Ap T' \Leftrightarrow \phi_1$. Now, by definition, we have $\Apm S=\neg (\Ap T' \triangleright \neg \Ap S')$. Since $S'$ is lower than $S$ and $\psi_0$ is $p^{\diamond}$-free, by $\mathcal{U}^{\circ}_p(S)$, we have $G \vdash (\; \Rightarrow \Ap S', \psi_0)$, and hence $G \vdash \neg \Ap S' \Rightarrow \psi_0$. As $G \vdash \Ap T' \Leftrightarrow \phi_1$ and $G \vdash \phi_1 \Leftrightarrow \phi_0$, we have $G \vdash \Ap T' \Leftrightarrow \phi_0$. By applying the rule $(CM)$ on the provable sequents $\neg \Ap S' \Rightarrow \psi_0$ and $\Ap T' \Leftrightarrow \phi_0$, we get $G \vdash \Ap T' \triangleright \neg \Ap S' \Rightarrow \phi_0 \triangleright \psi_0$, which implies
$G \vdash (\; \Rightarrow \neg (\Ap T' \triangleright \neg \Ap S'), \phi_0 \triangleright \psi_0)$ or equivalently $G \vdash (\bar{C} \Rightarrow \Apm S, \bar{D})$.\\
For the second case, let $\bar{D}=\varnothing$. If $\bar{C}=\varnothing$, then the sequent $S$ is provable which is a contradiction. Hence, $\bar{C}=\{\phi_1 \triangleright \psi_1\}$ which implies that $S=(\; \Rightarrow \phi_0 \triangleright \psi_0)$. 
Since $\bar{C}=\{\phi_1 \triangleright \psi_1\}$ and $\bar{C}$ is $p^{\circ}$-free, the formulas $\phi_1$ and $\psi_1$ are $p^{\diamond}$-free and $p^{\circ}$-free, respectively. Set $T''=(\phi_0 \Rightarrow \,)$ and $S''=(\, \Rightarrow \psi_0)$. Then, by $\mathcal{U}^{\circ}_p(S)$ and the fact that $T''$ is lower than $S$, we know that $\forall^{\circ}p T''$ exists and $G \vdash (\phi_0, \Ap T'' \Rightarrow \,)$ which implies $G \vdash \phi_0 \Rightarrow \neg \Ap T''$. As $G \vdash \phi_0 \Rightarrow \phi_{1}$ and $\phi_{1}$ is $p^{\diamond}$-free, we have $G \vdash (\, \Rightarrow \Ap T'', \phi_{1})$ which implies $G \vdash \neg \Ap T'' \Rightarrow \phi_1$. As $G \vdash \phi_0 \Leftrightarrow \phi_1$, we have $G \vdash \neg \Ap T'' \Leftrightarrow \phi_0$.
Now, by definition, we have $\Apm S=\neg \Ap T'' \triangleright \Ap S''$. As $S''$ is lower than $S$ and $\psi_1$ is $p^{\circ}$-free, by $\mathcal{U}^{\circ}_p(S)$, we have $G \vdash \psi_1 \Rightarrow \Ap S''$. As $G \vdash \neg \Ap T'' \Leftrightarrow \phi_0$ and $G \vdash \phi_0 \Leftrightarrow \phi_1$, we have $G \vdash \neg \Ap T'' \Leftrightarrow \phi_1$. Now, we can apply the rule $(CM)$ on the premises $\neg \Ap T'' \Leftrightarrow \phi_1$ and $\psi_1 \Rightarrow \Ap S''$ to prove $ \phi_{1} \triangleright \psi_{1} \Rightarrow \neg \Ap T'' \triangleright \Ap S''$ and hence $G \vdash (\bar{C} \Rightarrow \Apm S, \bar{D})$.\\

\noindent The case in which $G=\mathbf{GCMN}$ and the last rule is $(CN)$ is similar to the case $G = \mathbf{GCEN}$ where the last rule is $(CN)$, as discussed in Subsubsection \ref{CE}.

\subsubsection{Conditional logics $\mathsf{CMC}$ and $\mathsf{CK}$}

Let $G$ be either $\mathbf{GCMC}$ or $\mathbf{GCK}$. We will show that $G$ has MULIP. To define $\Apm S$, if  $\neg \, \mathcal{U}^{\circ}_p(S)$, set $\Apm S$ as $\bot$. If $\mathcal{U}^{\circ}_p(S) $, define $\Apm S$ as the following: if $S$ is provable, define $\Apm S=\top$. 
Otherwise, if $S$ is of the form $S=(\{\phi_k \triangleright \psi_k\}_{k \in K} \Rightarrow \,)$, for some non-empty multiset $K$ and there exists a $p^{\circ}$-free formula $\chi$ such that $V^{\dagger}(\chi) \subseteq V^{\dagger}(S)$, for any $\dagger \in \{+, -\}$ and $G \vdash \phi_k \Leftrightarrow \chi$, for any $k \in K$, then define $\Apm S=\neg (\chi \triangleright  \neg \Ap S')$, where $S'=(\{\psi_k\}_{k \in K} \Rightarrow \,)$. (The choice of $\chi$ is not important, as any two such formulas are equivalent. It is a simple consequence of the conditions $G \vdash \chi \Leftrightarrow \phi_k$, for any $k \in K$ and the non-emptiness of $K$).
If $S$ is of the form $S=(\{\phi_k \triangleright \psi_k\}_{k \in K} \Rightarrow \phi_0 \triangleright \psi_0)$, for some multiset $K$ (possibly empty), define $T''=(\phi_0 \Rightarrow \,)$ and $S''=(\{\psi_k\}_{k \in K} \Rightarrow \psi_0)$. Then, if $G \vdash \neg \Ap T'' \Leftrightarrow \phi_0 \Leftrightarrow \phi_k$, for any $k \in K$, define $\Apm S=\neg \Ap T'' \triangleright \Ap S''$. Otherwise, define $\Apm S=\bot$. Note that $\Apm S$ is well-defined, as in each case $S'$, $S''$ and $T''$ are lower than $S$, and we assumed $\mathcal{U}^{\circ}_p(S)$.\\
To show that $G$ has MULIP, we assume $\mathcal{U}^{\circ}_p(S)$ to prove the three conditions $(var)$, $(i)$ and $(ii')$ in Definition \ref{DfnUniformInterpolationSeq} for $\Apm S$. The condition $(var)$ is an immediate consequence of $p^{\circ}$-freeness of $\chi$ and the facts that $V^{\dagger}(\chi) \subseteq V^{\dagger}(S)$, for any $\dagger \in \{+, -\}$, the assumption $\mathcal{U}^{\circ}_p(S)$ and the fact that $S'$ and $S''$ and $T''$ are lower than $S$.\\
For $(i)$, if $S$ is provable, there is nothing to prove. If $S$ is of the form $S=(\{\phi_k \triangleright \psi_k\}_{k \in K} \Rightarrow \,)$, for some non-empty multiset $K$ and there is a $p^{\circ}$-free formula $\chi$ such that $V^{\dagger}(\chi) \subseteq V^{\dagger}(S)$, for any $\dagger \in \{+, -\}$ and $G \vdash \phi_k \Leftrightarrow \chi$, for any $k \in K$, then by definition we have $\Apm S=\neg (\chi \triangleright \neg \Ap S')$, where $S'=(\{\psi_k\}_{k \in K} \Rightarrow \,)$. As $K$ is non-empty, the sequent $S'$ is lower than $S$ and hence by $\mathcal{U}^{\circ}_p(S)$, we have $G \vdash (\{\psi_k\}_{k \in K}, \Ap S' \Rightarrow \;)$ or equivalently $G \vdash (\{\psi_k\}_{k \in K} \Rightarrow \neg \Ap S')$. As $G \vdash \chi \Leftrightarrow \phi_k$, for any $k \in K$ and $K$ is non-empty, by using the rule $(CMC)$, we get $G \vdash (\{\phi_k \triangleright \psi_k\}_{k \in K} \Rightarrow \chi \triangleright \neg \Ap S')$, which is equivalent to $G \vdash (\{\phi_k \triangleright \psi_k\}_{k \in K}, \neg (\chi \triangleright \neg \Ap S') \Rightarrow \,)$ and hence $G \vdash S \cdot (\Apm S \Rightarrow \,)$. \\
If $S$ is of the form $S=(\{\phi_k \triangleright \psi_k\}_{k \in K} \Rightarrow \phi_0 \triangleright \psi_0)$, for some multiset $K$ (possibly empty) and for any $k \in K$, we have $G \vdash \neg \Ap T'' \Leftrightarrow \phi_0 \Leftrightarrow \phi_k$, where $T''=(\phi_0 \Rightarrow \,)$, then by definition, we have $\Apm S=\neg \Ap T'' \triangleright \Ap S''$, where $S''=(\{\psi_k\}_{k \in K} \Rightarrow \psi_0)$. Using $\mathcal{U}^{\circ}_p(S)$ on $S''$, we have $G \vdash \{\psi_k\}_{k \in K}, \Ap S'' \Rightarrow \psi_0$. Since for any $k \in K$, we have $G \vdash \neg \Ap T'' \Leftrightarrow \phi_0 \Leftrightarrow \phi_k$, we can use the rule $(CMC)$ to show that $G \vdash (\{\phi_k \triangleright \psi_k\}_{k \in K}, \neg \Ap T'' \triangleright \Ap S'' \Rightarrow \phi_0 \triangleright \psi_0)$ and hence $G \vdash S  \cdot (\Apm S \Rightarrow \,)$. \\
For $(ii')$, let $S \cdot (\bar{C} \Rightarrow \bar{D})$ be derivable in $G$ and the last rule is the rule $(CMC)$, for a $p^{\diamond}$-free sequent $\bar{C} \Rightarrow \bar{D}$. We want to show that $\bar{C} \Rightarrow \Apm S, \bar{D}$ is derivable in $G$. If $S$ is provable, as $\Apm S=\top$, we have $\bar{C} \Rightarrow \Apm S, \bar{D}$. Therefore, we assume that $S$ is not provable.
As the last rule used in the proof of $S \cdot (\bar{C} \Rightarrow \bar{D})$ is $(CMC)$, the sequent must have the form $(\phi_1 \triangleright \psi_1 , \cdots , \phi_n \triangleright \psi_n \Rightarrow \phi_0 \triangleright \psi_0)$, for some $n \geq 1$ and the rule is in form:
\begin{center}
\begin{tabular}{c}
\AxiomC{$\{\phi_0 \Rightarrow \phi_i \;\;\; , \;\;\; \phi_i \Rightarrow \phi_0\}_{1 \leq i \leq n}$} 
 \AxiomC{$\psi_1, \cdots, \psi_n \Rightarrow \psi_0$} 
   \RightLabel{\footnotesize $CMC$} 
 \BinaryInfC{$\phi_1 \triangleright \psi_1 , \cdots , \phi_n \triangleright \psi_n \Rightarrow \phi_0 \triangleright \psi_0$}
 \DisplayProof
\end{tabular}
\end{center}
We know that $\bar{C}=\{\phi_{i} \triangleright \psi_{i}\}_{i \in M}$, for some $M \subseteq \{1, \cdots, n\}$. Hence, for each $i \in M$, the formulas $\phi_i$ and $\psi_{i}$ are $p^{\diamond}$-free and $p^{\circ}$-free, respectively.\\
There are two cases to consider, either $\bar{D}=\{\phi_0 \triangleright \psi_0\}$ or $\bar{D}=\varnothing$. First, assume $\bar{D}=\{\phi_0 \triangleright \psi_0\}$. Then, the formulas $\phi_0$ and $\psi_0$ are $p^{\circ}$-free and $p^{\diamond}$-free, respectively. Now, if $M=\{1, \cdots, n\}$, then $\bar{C} \Rightarrow \bar{D}$ and hence $\bar{C} \Rightarrow \bar{D}, \Apm S$ are both provable in $G$. Therefore, assume $M \neq \{1, \cdots, n\}$ and set $a \in \{1, \cdots, n\}-M$. Note that $\phi_a \triangleright \psi_a$ occurs in the antecedent of $S$. By $\mathcal{U}^{\circ}_p(S)$ and the fact that $(\; \Rightarrow \phi_a)$ is lower than $S$, we know that $\forall^{\circ}p (\; \Rightarrow \phi_a)$ exists. Set $\chi=\forall^{\circ}p (\; \Rightarrow \phi_a)$. By $\mathcal{U}^{\circ}_p(S)$, we know that $V^{\dagger}(\chi) \subseteq V^{\dagger}(S)$ and $G \vdash \chi \Rightarrow \phi_a$. Moreover, since $\phi_0$ is $p^{\circ}$-free and $G \vdash \phi_0 \Rightarrow \phi_a$, we have $G \vdash \phi_0 \Rightarrow \forall^{\circ}p (\; \Rightarrow \phi_a)$, or equivalently $G \vdash \phi_0 \Rightarrow \chi$. As $G \vdash \phi_0 \Leftrightarrow \phi_i$, for any $i \in \{1, \cdots, n\}$, we have $G \vdash \phi_i \Leftrightarrow \chi$, for any $i \in \{1, \cdots, n\}$. Now, by the form of $S$ and the fact that $S^a$ is non-empty, we have $\Apm S=\neg (\chi \triangleright \neg \Ap S')$, where $S'=(\{\psi_i\}_{i \notin M} \Rightarrow \,)$. 
Since $M \neq \{1, \cdots , n\}$ and hence $\{\psi_i\}_{i \notin M}$ is not empty, we know that $S'$ is lower than $S$. Moreover,
$(\{\psi_{i}\}_{i \in M} \Rightarrow \psi_0)$ is $p^{\diamond}$-free and hence by $\mathcal{U}^{\circ}_p(S)$, we have $G \vdash (\{\psi_{i}\}_{i \in M} \Rightarrow \Ap S', \psi_0)$. Now, we can apply the rule $(CMC)$ 
\begin{center}
\begin{tabular}{c c}
\AxiomC{$\phi_0 \Leftrightarrow \chi \;\;\;,\;\;\; \{\phi_0 \Leftrightarrow \phi_i\}_{i \in M}$} 
 \AxiomC{$ \neg \Ap S', \{\psi_i\}_{i \in M} \Rightarrow \psi_0$} 
   \RightLabel{\footnotesize$CMC$} 
 \BinaryInfC{$\chi \triangleright \neg \Ap S', \{\phi_i \triangleright \psi_i\}_{i \in M} \Rightarrow \phi_0 \triangleright \psi_0$}
 \DisplayProof
\end{tabular}
\end{center}
and equivalently, we have $G \vdash (\bar{C} \Rightarrow \Apm S, \bar{D})$.\\
For $\bar{D}=\varnothing$, if $M=\varnothing$, then $S$ would be provable which contradicts our assumption. Hence, $M \neq \varnothing$. Set $T''=(\phi_{0} \Rightarrow \;)$ and $S''=(\{\psi_i\}_{i \notin M} \Rightarrow \psi_0)$. By $\mathcal{U}^{\circ}_p(S)$ and the fact that $T''$ is lower than $S$, we know that $\forall^{\circ}p T''$ exists and $G \vdash (\phi_0, \forall^{\circ}p T'' \Rightarrow \,)$ which is equivalent to $G \vdash \phi_0 \Rightarrow \neg \forall^{\circ}p T''$. For any $i \in M$, we have $G \vdash \phi_0 \Rightarrow \phi_{i}$ and as $\phi_{i}$ is $p^{\diamond}$-free, by $\mathcal{U}^{\circ}_p(S)$, we have $G \vdash (\, \Rightarrow \forall^{\circ}p T'', \phi_i)$ or equivalently $G \vdash (\neg \forall^{\circ}p T'' \Rightarrow \phi_i)$. Moreover, for any $i \in M$, we have $G \vdash \phi_i \Leftrightarrow \phi_0$. Since $M$ is non-empty, we have $G \vdash \neg \forall^{\circ}p T'' \Leftrightarrow \phi_0 \Leftrightarrow \phi_i$ for any $i \in M$ and hence for any $i \in \{1, \cdots, n\}$.\\
Now, by definition, we have $\Apm S=\neg \Ap T'' \triangleright \Ap S''$. As $S''$ is lower than $S$, by $\mathcal{U}^{\circ}_p(S)$ and the fact that $(\{\psi_i\}_{i \in M} \Rightarrow \,)$ is $p^{\diamond}$-free, we have $ G \vdash \{\psi_i\}_{i \in M} \Rightarrow \Ap S''$. As $M$ is non-empty, we can apply the rule $(CMC)$ 
\begin{center}
\begin{tabular}{c c}
\AxiomC{$\{\phi_i \Leftrightarrow \neg \Ap T''\}_{i \in M}$} 
 \AxiomC{$ \{\psi_i\}_{i \in M} \Rightarrow \Ap S''$} 
   \RightLabel{\footnotesize$CMC$} 
 \BinaryInfC{$\{\phi_i \triangleright \psi_i\}_{i \in M} \Rightarrow \neg \Ap T'' \triangleright \Ap S''$}
 \DisplayProof
\end{tabular}
\end{center}
and hence $G \vdash \bar{C} \Rightarrow \Apm S, \bar{D}$.\\

\noindent The case in which $G=\mathbf{GCK}$ and the last rule is $(CN)$ is similar to the case $G = \mathbf{GCEN}$ where the last rule is $(CN)$, as discussed in Subsubsection \ref{CE}.

\subsubsection{Conditional logic $\mathsf{CKID}$}
Let $G$ be $\mathbf{GCKID}$. We will show that $G$ has MULIP. To define $\Apm S$, if  $\neg \, \mathcal{U}^{\circ}_p(S)$, define $\Apm S$ as $\bot$. If $\mathcal{U}^{\circ}_p(S) $, define $\Apm S$ as the following: if $S$ is provable, define $\Apm S=\top$. Otherwise, if $S$ is of the form $S=(\{\phi_k \triangleright \psi_k\}_{k \in K} \Rightarrow \,)$, for some non-empty multiset $K$ and there exists $p^{\circ}$-free formula $\chi$ such that $V^{\dagger}(\chi) \subseteq V^{\dagger}(S)$, for any $\dagger \in \{+, -\}$ and $G \vdash \phi_k \Leftrightarrow \chi$, for any $k \in K$, then define $\Apm S=\neg (\chi \triangleright  \neg \Ap S')$, where $S'=(\{\psi_k\}_{k \in K} \Rightarrow \,)$. (The choice of $\chi$ is not important, as any two such formulas are equivalent. It is a simple consequence of the conditions $G \vdash \chi \Leftrightarrow \phi_k$, for any $k \in K$ and the non-emptiness of $K$). If $S$ is of the form $S=(\{\phi_k \triangleright \psi_k\}_{k \in K} \Rightarrow \phi_0 \triangleright \psi_0)$, define
$T''=(\phi_0 \Rightarrow \,)$ and $S''=(\phi_0, \{\psi_k\}_{k \in K} \Rightarrow \psi_0)$. Then, if $G \vdash \neg \Ap T'' \Leftrightarrow \phi_0 \Leftrightarrow \phi_k$, for any $k \in K$, define $\Apm S=\neg \Ap T'' \triangleright \Ap S''$. Otherwise, define $\Apm S=\bot$. Note that $\Apm S$ is well-defined as we assumed $\mathcal{U}^{\circ}_p(S)$ and in each case $S'$, $S''$ and $T''$ are lower than $S$.\\
To show that $G$ has MULIP, we assume $\mathcal{U}^{\circ}_p(S)$ to prove the three conditions $(var)$, $(i)$ and $(ii')$ in Definition \ref{DfnUniformInterpolationSeq} for $\Apm S$. The condition $(var)$ is an immediate consequence of $p^{\circ}$-freeness of $\chi$ and the fact that $V^{\dagger}(\chi) \subseteq V^{\dagger}(S)$, for any $\dagger \in \{+, -\}$ and the facts that $\mathcal{U}^{\circ}_p(S)$ holds and $S'$ and $S''$ and $T''$ are lower than $S$.\\
For $(i)$, if $S$ is provable, there is nothing to prove. If $S$ is of the form $S=(\{\phi_k \triangleright \psi_k\}_{k \in K} \Rightarrow \,)$, for some non-empty $K$ and there is $p^{\circ}$-free formula $\chi$ such that $V^{\dagger}(\chi) \subseteq V^{\dagger}(S)$, for any $\dagger \in \{+, -\}$ and $G \vdash \phi_k \Leftrightarrow \chi$, for any $k \in K$, then by definition, we have $\Apm S=\neg (\chi \triangleright \neg \Ap S')$, where $S'=(\{\psi_k\}_{k \in K} \Rightarrow \,)$. Since we assumed that $K$ is non-empty, we know that $S'$ is lower than $S$. Therefore, by
$\mathcal{U}^{\circ}_p(S)$, we have $(\{\psi_k\}_{k \in K}, \Ap S' \Rightarrow \;)$ or equivalently $(\{\psi_k\}_{k \in K} \Rightarrow \neg \Ap S')$ and then using the weakening rule, we get $\chi, \{\psi_k\}_{k \in K} \Rightarrow \neg \Ap S'$. As $G \vdash \chi \Leftrightarrow \phi_k$, for any $k \in K$, by using 
the rule $(CKID)$, we get $G \vdash (\{\phi_k \triangleright \psi_k\}_{k \in K} \Rightarrow \chi \triangleright \neg \Ap S')$, which is equivalent to $G \vdash (\{\phi_k \triangleright \psi_k\}_{k \in K} , \neg (\chi \triangleright \neg \Ap S') \Rightarrow \,)$ and hence $G \vdash S \cdot (\Apm S \Rightarrow \,)$.\\
If $S$ is of the form $S=(\{\phi_k \triangleright \psi_k\}_{k \in K} \Rightarrow \phi_0 \triangleright \psi_0)$ and $G \vdash \neg \Ap T'' \Leftrightarrow \phi_0 \Leftrightarrow \phi_k$, for $T''=(\phi_0 \Rightarrow \,)$ and any $k \in K$, then we have $\Apm S=\neg \Ap T'' \triangleright \Ap S''$, where $S''=(\phi_0, \{\psi_k\}_{k \in K} \Rightarrow \psi_0)$. Using $\mathcal{U}^{\circ}_p(S)$ on $S''$, we have $G \vdash \phi_0, \{\psi_k\}_{k \in K}, \Ap S'' \Rightarrow \psi_0$. Since for any $k \in K$, we have $G \vdash \neg \Ap T'' \Leftrightarrow \phi_0 \Leftrightarrow \phi_k$, we can use the rule $(CKID)$ to show that $G \vdash (\{\phi_k \triangleright \psi_k\}_{k \in K}, \neg \Ap T'' \triangleright \Ap S'' \Rightarrow \phi_0 \triangleright \psi_0)$ and hence $G \vdash S  \cdot (\Apm S \Rightarrow \,)$.\\
For $(ii')$, let $S \cdot (\bar{C} \Rightarrow \bar{D})$ be derivable in $G$ and the last rule is the rule $(CKID)$, for a $p^{\diamond}$-free sequent $\bar{C} \Rightarrow \bar{D}$. We want to show that $\bar{C} \Rightarrow \Apm S, \bar{D}$ is derivable in $G$. If $S$ is provable, as $\Apm S=\top$, we have $\bar{C} \Rightarrow \Apm S, \bar{D}$. Therefore, we assume that $S$ is not provable.
As the last rule used in the proof of $S \cdot (\bar{C} \Rightarrow \bar{D})$ is $(CKID)$, the sequent must have the form $(\{\phi_i \triangleright \psi_i\}_{i \in I} \Rightarrow \phi_0 \triangleright \psi_0)$ for some (possibly empty) multiset $I$ and the rule has the form:
\begin{center}
\begin{tabular}{c}
\AxiomC{$\{\phi_0 \Rightarrow \phi_i \;\;\; , \;\;\; \phi_i \Rightarrow \phi_0\}_{i \in I}$} 
 \AxiomC{$\phi_0, \{\psi_i\}_{i \in I} \Rightarrow \psi_0$} 
   \RightLabel{\footnotesize$CKID$} 
 \BinaryInfC{$\{\phi_i \triangleright \psi_i\}_{i \in I} \Rightarrow \phi_0 \triangleright \psi_0$}
 \DisplayProof
\end{tabular}
\end{center}
Let $\bar{C}=\{\phi_i \triangleright \psi_i\}_{i \in M}$, for some $M \subseteq I$. Hence, for any $i \in M$, the formulas $\phi_i$ and $\psi_{i}$ are $p^{\diamond}$-free and $p^{\circ}$-free, respectively.\\ 
There are two cases to consider, either $\bar{D}=\{\phi_0 \triangleright \psi_0\}$ or $\bar{D}=\varnothing$. First, assume $\bar{D}=\{\phi_0 \triangleright \psi_0\}$. Then, $\phi_0$ and $\psi_0$ are $p^{\circ}$-free and $p^{\diamond}$-free, respectively. If $M=I$, then $G \vdash \bar{C} \Rightarrow \bar{D}$ and then by weakening, we have $G \vdash \bar{C} \Rightarrow \bar{D}, \Apm S$. Hence, assume $M \neq I$ and set $a \in I-M$. Note that $\phi_a \triangleright \psi_a$ occurs in the antecedent of $S$. By $\mathcal{U}^{\circ}_p(S)$ and the fact that the sequent $(\; \Rightarrow \phi_a)$ is lower than $S$, we know that $ \forall^{\circ}p (\; \Rightarrow \phi_a)$ exists. Set $\chi=\forall^{\circ}p (\; \Rightarrow \phi_a)$. By $\mathcal{U}^{\circ}_p(S)$, we know that $V^{\dagger}(\chi) \subseteq V^{\dagger}(S)$, $G \vdash \chi \Rightarrow \phi_a$ and since $\phi_0$ is $p^{\circ}$-free and $G \vdash \phi_0 \Rightarrow \phi_a$, we have $G \vdash \phi_0 \Rightarrow \chi$. As $G \vdash \phi_0 \Leftrightarrow \phi_i$, for any $i \in I$, we have $G \vdash \phi_i \Leftrightarrow \chi$. Now, by the form of $S$ and the fact that $S^{a}$ is non-empty (since $M \neq I$), by definition, we have $\Apm S=\neg (\chi \triangleright \neg \Ap S')$, where $S'=(\{\psi_i\}_{i \in I-M} \Rightarrow \,)$. Again, by $M \neq I$, we know that $S'$ is lower than $S$. As $G \vdash (\phi_0, \{\psi_i\}_{i \in I} \Rightarrow \psi_0)$ or equivalently $G \vdash S' \cdot (\phi_0, \{\psi_i\}_{i \in M} \Rightarrow \psi_0)$ and $(\phi_0, \{\psi_i\}_{i \in M} \Rightarrow \psi_0)$ is $p^{\diamond}$-free, by $\mathcal{U}^{\circ}_p(S)$, we have $G \vdash (\phi_0, \{\psi_i\}_{i \in M} \Rightarrow \Ap S', \psi_0)$, or equivalently $G \vdash (\phi_0, \{\psi_i\}_{i \in M}, \neg \Ap S' \Rightarrow \psi_0)$. Applying the rule $(CKID)$ on the latter sequent and the provable sequents $\{\phi_i \Leftrightarrow \phi_0\}_{i \in M}$ and $\chi \Leftrightarrow \phi_0$, we get  $G \vdash (\{\phi_i \triangleright \psi_i\}_{i \in M}, (\chi \triangleright \neg \Ap S') \Rightarrow \phi_0 \triangleright \psi_0)$. Therefore, $G \vdash (\{\phi_i \triangleright \psi_i\}_{i \in M} \Rightarrow \neg (\chi \triangleright \neg \Ap S'), \phi_0 \triangleright \psi_0)$ or equivalently $G \vdash (\bar{C} \Rightarrow \Apm S, \bar{D})$. \\
If $\bar{D}=\varnothing$, then $\bar{C} \neq \varnothing$. Otherwise, $S$ would be provable which contradicts our assumption. Therefore, $M \neq \varnothing$. Moreover, note that $S=(\{\phi_i \triangleright \psi_i\}_{i \in I-M} \Rightarrow \phi_0 \triangleright \psi_0)$. Set $T''=(\phi_{0} \Rightarrow \;)$ and $S''=(\phi_0, \{\psi_i\}_{i \in I -M} \Rightarrow \psi_0)$. By $\mathcal{U}^{\circ}_p(S)$ and the fact that $T''$ is lower than $S$, we know that $\forall^{\circ}p T''$ exists and $G \vdash (\phi_0, \forall^{\circ}p T'' \Rightarrow \,)$ which is equivalent to $G \vdash \phi_0 \Rightarrow \neg \forall^{\circ}p T''$. We know $G \vdash \phi_i \Leftrightarrow \phi_0$, for any $i \in I$. Specially, for any $i \in M$, we have $G \vdash \phi_0 \Rightarrow \phi_{i}$ and as $\phi_{i}$ is $p^{\diamond}$-free, by $\mathcal{U}^{\circ}_p(S)$, we have $G \vdash (\, \Rightarrow \forall^{\circ}p T'', \phi_i)$ or equivalently $G \vdash (\neg \forall^{\circ}p T'' \Rightarrow \phi_i)$. Since $G \vdash \phi_i \Leftrightarrow \phi_0$, for any $i \in M$ and $M$ is non-empty, we have $G \vdash  \neg \forall^{\circ}p T''  \Leftrightarrow \phi_0 \Leftrightarrow \phi_i$, for any $i \in M$ and hence for an any $i \in I$. 
Now, by definition, we have $\Apm S=\neg \Ap T'' \triangleright \Ap S''$. Note that $S''$ is lower than $S$ and $G \vdash (\phi_0, \{\psi_i\}_{i \in I} \Rightarrow \psi_0)$ or equivalently $G \vdash S'' \cdot ( \{\psi_i\}_{i \in M} \Rightarrow )$. Since $(\{\psi_i\}_{i \in M} \Rightarrow \;)$ is $p^{\diamond}$-free, by $\mathcal{U}^{\circ}_p(S)$, we have $G \vdash \{\psi_i\}_{i \in M} \Rightarrow \Ap S''$. Therefore, using the weakening rule, we have $G \vdash \neg \Ap T'', \{\psi_i\}_{i \in M} \Rightarrow \Ap S''$. By applying the rule $(CKID)$ on the latter sequent and the provable sequents $\{\neg \Ap T'' \Leftrightarrow \phi_i\}_{i \in M}$, we get $G \vdash \{\phi_{i} \triangleright \psi_{i}\}_{i \in M} \Rightarrow \neg \Ap T''  \triangleright \Ap S''$ and hence $G \vdash (\bar{C} \Rightarrow \Apm S, \bar{D})$.
\subsection{Uniform Interpolation} \label{SecUIP}
In this subsection, we will prove that the logics $\mathsf{CKCEM}$ and $\mathsf{CKCEMID}$ have uniform interpolation. Surprisingly, these two logics do not enjoy uniform Lyndon interpolation, as will be shown in Section \ref{SecNegativeResults}.

\begin{theorem}
The logics $\mathsf{CKCEM}$ and $\mathsf{CKCEMID}$ enjoy UIP and hence CIP.
\end{theorem}
\begin{proof}
Using Theorem \ref{GeneralArgument} and Theorem \ref{SequentImpliesLogic}, it is enough to show that the sequent calculi $\mathbf{GCKCEM}$ and $\mathbf{GCKCEMID}$ have MUIP.
We only investigate the case $G=\mathbf{GCKCEM}$. The case $G=\mathbf{GCKCEMID}$ is similar. To define $\Apmcf S$, if  $\neg \, \mathcal{U}_p(S)$, define $\Apmcf S$ as $\bot$. If $\mathcal{U}_p(S) $, define $\Apmcf S$ as the following: if $S$ is provable, define $\Apmcf S=\top$. Otherwise, if $S$ is of the form $S=(\{\phi_k \triangleright \psi_k\}_{k \in K} \Rightarrow \,)$, for some non-empty multiset $K$ and if there exists a $p$-free formula $\chi$ such that $V(\chi) \subseteq V(S)$ and $G \vdash \phi_k \Leftrightarrow \chi$, for any $k \in K$, then define $\Apmcf S=\neg (\chi \triangleright \neg \forall p S')$, where $S'=(\{\psi_k\}_{k \in K} \Rightarrow \,)$. (The choice of $\chi$ is not important, as any two such formulas are equivalent. It is a simple consequence of the conditions $G \vdash \chi \Leftrightarrow \phi_k$, for any $k \in K$ and the non-emptiness of $K$). If $S$ is of the form $S=(\{\phi_k \triangleright \psi_k\}_{k \in K} \Rightarrow \{\phi_l \triangleright \psi_l\}_{l \in L})$, for some multisets $K$ and $L$ such that $L$ is non-empty and if there exists a $p$-free formula $\chi$ such that $V(\chi) \subseteq V(S)$ and $G \vdash \phi_r \Leftrightarrow \chi$, for any $r \in K \cup L$, then define $\Apmcf S=(\chi \triangleright \forall p S')$, where $S'=(\{\psi_k\}_{k \in K} \Rightarrow \{\psi_l\}_{l \in L})$. (Note that the choice of $\chi$ is not important, again. This time, we use the non-emptiness of $L$). Otherwise, define $\Apmcf S=\bot$. Note that $\Apmcf S$ is well-defined as $S'$ is lower than $S$ and we assumed $\mathcal{U}_p(S)$.\\
To show that $G$ has MUIP, we assume $\mathcal{U}_p(S)$ to prove the three conditions $(var)$, $(i)$ and $(ii')$ in Definition \ref{DfnUniformInterpolationSeq} for $\Apmcf S$. The condition $(var)$ is an immediate consequence of $p$-freeness of $\chi$ and the fact that $V(\chi) \subseteq V(S)$, the assumption $\mathcal{U}_p(S)$ and the fact that $S'$ is lower than $S$.\\
For $(i)$, if $S$ is of the form $S=(\{\phi_k \triangleright \psi_k\}_{k \in K} \Rightarrow \,)$, where $K \neq \varnothing$, and there exists a $p$-free formula $\chi$ such that $V(\chi) \subseteq V(S)$ and $G \vdash \phi_k \Leftrightarrow \chi$, for any $k \in K$, then by definition, we have $\Apmcf S=\neg (\chi \triangleright \neg \forall p S')$, where $S'=(\{\psi_k\}_{k \in K} \Rightarrow \,)$. As $K$ is non-empty, the sequent $S'$ is lower than $S$. Hence, using $\mathcal{U}_p(S)$ on $S'$, we have $G \vdash (\{\psi_k\}_{k \in K}, \forall p S' \Rightarrow \,)$, or equivalently $G \vdash (\{\psi_k\}_{k \in K} \Rightarrow \neg \forall p S')$. Applying the rule $(CKCEM)$ on the latter sequent and the provable sequents $\{\chi \Leftrightarrow \phi_k\}_{k \in K}$, we get $G \vdash \{\phi_k \triangleright \psi_k\}_{k \in K} \Rightarrow \chi \triangleright \neg \forall p S'$ or equivalently $G \vdash (\{\phi_k \triangleright \psi_k\}_{k \in K}, \neg (\chi \triangleright \neg \forall p S') \Rightarrow \,)$. Therefore, we have $G \vdash S  \cdot (\Apmcf S \Rightarrow \,)$.\\
If $S$ is of the form $S=(\{\phi_k \triangleright \psi_k\}_{k \in K} \Rightarrow \{\phi_l \triangleright \psi_l\}_{l \in L})$, where $L \neq \varnothing$, and there exists a $p$-free formula $\chi$ such that $V(\chi) \subseteq V(S)$ and $G \vdash \phi_r \Leftrightarrow \chi$, for any $r \in K \cup L$, then by definition, we have $\Apmcf S=(\chi \triangleright \forall p S')$, where $S'=(\{\psi_k\}_{k \in K} \Rightarrow \{\psi_l\}_{l \in L})$. As $L$ is non-empty, the sequent $S'$ is lower than $S$. Hence, using $\mathcal{U}_p(S)$ on $S'$, we have $G \vdash \{\psi_k\}_{k \in K}, \forall p S' \Rightarrow \{\psi_l\}_{l \in L}$. Since $L$ is non-empty and $G \vdash \phi_r \Leftrightarrow \chi$, for any $r \in K \cup L$, applying the rule $(CKCEM)$ we get $G \vdash \{\phi_k \triangleright \psi_k\}_{k \in K}, \chi \triangleright \forall p S' \Rightarrow  \{\phi_l \triangleright \psi_l\}_{l \in L}$. Therefore, we have $G \vdash S  \cdot (\Apmcf S \Rightarrow \,)$.\\
For $(ii')$, let $S \cdot (\bar{C} \Rightarrow \bar{D})$ be derivable in $G$ and the last rule be the rule $(CKCEM)$, for a $p$-free sequent $\bar{C} \Rightarrow \bar{D}$. We want to show that $\bar{C} \Rightarrow \Apmcf S, \bar{D}$ is also derivable in $G$. If $S$ is provable, as $\Apmcf S=\top$, we have $\bar{C} \Rightarrow \Apmcf S, \bar{D}$. Therefore, we assume that $S$ is not provable.
As the last rule used in the proof of $S \cdot (\bar{C} \Rightarrow \bar{D})$ is $(CMCEM)$, the sequent must have the form $(\{\phi_i \triangleright \psi_i\}_{i \in I} \Rightarrow \phi_0 \triangleright \psi_0, \{\phi_j \triangleright \psi_j\}_{j \in J})$, for some multisets $I$ and $J$ (possibly empty) and the rule is in the form:
\begin{center}
\begin{tabular}{c c c}
\AxiomC{$\{\phi_0 \Rightarrow \phi_r \;\;\; , \;\;\; \phi_r \Rightarrow \phi_0\}_{r \in I \cup J}$} 
 \AxiomC{$\{\psi_i\}_{i \in I} \Rightarrow \psi_{0}, \{\psi_j\}_{j \in J}$} 
   \RightLabel{\footnotesize$CKCEM$} 
 \BinaryInfC{$\{\phi_i \triangleright \psi_i\}_{i \in I} \Rightarrow \phi_{0} \triangleright \psi_{0}, \{\phi_j \triangleright \psi_j\}_{j \in J}$}
 \DisplayProof
\end{tabular}
\end{center}
Let $\bar{C}=\{\phi_i \triangleright \psi_i\}_{i \in M}$, for some (possibly empty) $M \subseteq I$. Hence, for any $i \in M$, the formulas $\phi_i$ and $\psi_{i}$ are $p$-free.
There are two cases to consider, either $\phi_0 \triangleright \psi_0 \in \bar{D}$ or $\phi_0 \triangleright \psi_0 \notin \bar{D}$. First, assume $\phi_0 \triangleright \psi_0 \in \bar{D}$. Then, $\bar{D}=\{\phi_0 \triangleright \psi_0\} \cup \{\phi_j \triangleright \psi_j\}_{j \in N}$, for some (possibly empty) $N \subseteq J$. Hence, for any $j \in N$, the formulas $\phi_j$ and $\psi_{j}$ and $\phi_0$ and $\psi_0$ are $p$-free.
If $M=I$ and $N=J$, then the sequent $\bar{C} \Rightarrow \bar{D}$ and hence $\bar{C} \Rightarrow \Apmcf S, \bar{D}$ is provable. Hence, assume either $M \neq I$ or $N \neq J$. Therefore, $M \cup N \neq I \cup J$.
Pick $a \in (I \cup J)-(M \cup N)$ and note that $\phi_a \triangleright \psi_a$ occurs in $S$.
By $\mathcal{U}_p(S)$ and the fact that $(\; \Rightarrow \phi_a)$ is lower than $S$, we know that $\forall p (\; \Rightarrow \phi_a)$ exists. Set $\chi=\forall p (\; \Rightarrow \phi_a)$. By $\mathcal{U}_p(S)$, we know that $V(\chi) \subseteq V(S)$ and $G \vdash \chi \Rightarrow  \phi_a$. Moreover, for the sequent $(\phi_0 \Rightarrow \,)$, since $\phi_0$ is $p$-free and $G \vdash \phi_0 \Rightarrow \phi_a$, by $\mathcal{U}_p(S)$, we have $G \vdash \phi_0 \Rightarrow \chi$.  As $G \vdash \phi_0 \Leftrightarrow \phi_r$, for any $r \in I \cup J$, we have $G \vdash \phi_r \Leftrightarrow \chi$, for any $r \in I \cup J$. Now, there are two cases to consider. Either $N=J$ or $N \neq J$. In the first case, we have $M \neq I$. Hence, $S=(\{\phi_i \triangleright \psi_i\}_{i \in I-M} \Rightarrow \,)$. Therefore, by definition and the fact that $I-M$ is non-empty, we have $\Apmcf S=\neg (\chi \triangleright \neg \forall p S')$, where $S'=(\{\psi_i\}_{i \in I-M} \Rightarrow \,)$. We know that $G \vdash \{\psi_i\}_{i \in I} \Rightarrow \psi_{0}, \{\psi_j\}_{j \in J}$ or equivalently $G \vdash S' \cdot (\{\psi_i\}_{i \in M} \Rightarrow \psi_{0}, \{\psi_j\}_{j \in N})$. As $S'$ is lower than $S$, by $\mathcal{U}_p(S)$, we have $G \vdash \{\psi_{i}\}_{i \in M} \Rightarrow \forall p S', \psi_0, \{\psi_{j}\}_{j \in N}$, or equivalently, $G \vdash \{\psi_{i}\}_{i \in M}, \neg \forall p S' \Rightarrow \psi_0, \{\psi_{j}\}_{j \in N}$. Applying the rule $(CMCEM)$ on the latter sequent and on the provable sequents $\chi \Leftrightarrow \phi_0$ and $\{\phi_0 \Leftrightarrow \phi_r\}_{r \in M \cup N}$, we get $G \vdash \{\phi_i \triangleright \psi_{i}\}_{i \in M}, \chi \triangleright \neg \forall p S' \Rightarrow  \phi_0 \triangleright \psi_0, \{\phi_j \triangleright \psi_{j}\}_{j \in N}$ or equivalently $G \vdash \{\phi_i \triangleright \psi_{i}\}_{i \in M} \Rightarrow \neg (\chi \triangleright \neg \forall p S'), \phi_0 \triangleright \psi_0, \{\phi_j \triangleright \psi_{j}\}_{j \in N}$. Therefore, we have $G \vdash (\bar{C} \Rightarrow \Apmcf S, \bar{D})$.\\ 
In the second case, we have $N \neq J$. Hence, $S=(\{\phi_i \triangleright \psi_i\}_{i \in I-M} \Rightarrow \{\phi_j \triangleright \psi_j\}_{j \in J-N})$. By definition and the fact that $J-N$ is non-empty, we have $\Apmcf S=(\chi \triangleright \forall p S')$, where $S'=(\{\psi_i\}_{i \in I-M} \Rightarrow \{\psi_j\}_{j \in J-N})$. We know that $G \vdash \{\psi_i\}_{i \in I} \Rightarrow \psi_{0}, \{\psi_j\}_{j \in J}$ or equivalently $G \vdash S' \cdot (\{\psi_i\}_{i \in M} \Rightarrow \psi_{0}, \{\psi_j\}_{j \in N})$. As $N \neq J$, we know that $S'$ is lower than $S$, and hence by $\mathcal{U}_p(S)$, we have $G \vdash \{\psi_{i}\}_{i \in M} \Rightarrow \forall p S', \psi_0, \{\psi_{j}\}_{j \in N}$. Applying the rule $(CMCEM)$ on the latter sequent and on the provable sequents $\chi \Leftrightarrow \phi_0$ and $\{\phi_0 \Leftrightarrow \phi_r\}_{r \in M \cup N}$, we get $G \vdash \{\phi_i \triangleright \psi_{i}\}_{i \in M} \Rightarrow (\chi \triangleright \forall p S'), \phi_0 \triangleright \psi_0, \{\phi_j \triangleright \psi_{j}\}_{j \in N}$. Therefore, we have $G \vdash (\bar{C} \Rightarrow \Apmcf S, \bar{D})$.\\ 
If $\phi_0 \triangleright \psi_0 \notin \bar{D}$, then, $\bar{D}=\{\phi_j \triangleright \psi_j\}_{j \in N}$, for some (possibly empty) $N \subseteq J$. Hence, for any $j \in N$, the formulas $\phi_j$ and $\psi_{j}$ are $p$-free. If $M \cup N=\varnothing$, then $S$ must be provable which is a contradiction. Hence, $M \cup N$ is non-empty. 
As $\phi_0 \triangleright \psi_0$ occurs in $S$, the sequent $(\phi_0 \Rightarrow \,)$ is lower than $S$. By $\mathcal{U}_p(S)$, we know that $\forall p (\phi_0 \Rightarrow \,)$ exists and $G \vdash (\phi_0, \forall p (\phi_0 \Rightarrow \,) \Rightarrow \;)$. Set $\chi=\neg \forall p (\phi_0 \Rightarrow \,)$. Hence, $G \vdash \phi_0 \Rightarrow \chi$. Moreover, for any $r \in M \cup N$, the sequent $( \Rightarrow \phi_{r})$ is $p$-free and $G \vdash \phi_0 \Rightarrow \phi_{r}$. Therefore, since $(\phi_0 \Rightarrow \,)$ is lower than $S$, by $\mathcal{U}_p(S)$, we have $G \vdash \; \Rightarrow \phi_r, \forall p(\phi_0 \Rightarrow)$, or equivalently $G \vdash \chi \Rightarrow \phi_r$, for any $r \in M \cup N$.
Since $G \vdash \phi_r \Leftrightarrow \phi_0$, for any $r \in I \cup J$ and $M \cup N$ is non-empty, we have $G \vdash \chi \Leftrightarrow \phi_0 \Leftrightarrow \phi_r$, for any $r \in M \cup N$ and hence for any $r \in I \cup J$.
Now, as $S=(\{\phi_i \triangleright \psi_i\}_{i \in I-M} \Rightarrow \phi_0 \triangleright \psi_0, \{\phi_j \triangleright \psi_j\}_{j \in J-N})$, by definition, we have $\Apmcf S=(\chi \triangleright \forall p S')$, where $S'=(\{\psi_i\}_{i \in I-M} \Rightarrow \psi_0, \{\psi_j\}_{j \in J-N})$. Again, we know that $G \vdash \{\psi_i\}_{i \in I} \Rightarrow \psi_{0}, \{\psi_j\}_{j \in J}$ or equivalently $G \vdash S'\cdot (\{\psi_i\}_{i \in M} \Rightarrow \{\psi_j\}_{j \in N})$. As $M \cup N$ is non-empty, $S'$ is lower than $S$, and hence by $\mathcal{U}_p(S)$ we have $G \vdash \{\psi_i\}_{i \in M} \Rightarrow \forall p S',  \{\psi_j\}_{j \in N}$. By applying the rule $(CKCEM)$, we can prove $G \vdash \{\phi_{i} \triangleright \psi_{i}\}_{i \in M} \Rightarrow (\chi  \triangleright \forall p S'), \{\phi_{j} \triangleright \psi_{j}\}_{j \in N}$ and hence $G \vdash (\bar{C} \Rightarrow \Apmcf S, \bar{D})$.
\end{proof}

\section{Negative Results} \label{SecNegativeResults}

In this section, we present some negative results on the failure of some interpolation properties for a number of modal and conditional logics. First, we prove that none of the logics $\mathsf{EC}$, $\mathsf{ECN}$, $\mathsf{CEC}$ and $\mathsf{CECN}$ enjoy Craig interpolation. Then, we show that for the extensions of $\mathsf{CKCEM}$, uniform Lyndon interpolation property leads to some peculiar behavior of the conditional. As a consequence, we show that the logics $\mathsf{CKCEM}$ and $\mathsf{CKCEMID}$ do not enjoy uniform Lyndon interpolation property. \\

We first investigate the modal logics $\mathsf{EC}$ and $\mathsf{ECN}$. 
\begin{theorem}\label{Thm:CIP EC ECN}
Logics $\mathsf{EC}$ and $\mathsf{ECN}$ do not have CIP. As a consequence, they do not have UIP or ULIP.
\end{theorem}
\begin{proof}
Set $\phi=\Box (\neg q \wedge r)$ and $\psi= \Box (p \wedge q) \to \Box \bot$, where $p$, $q$, and $r$ are three 
distinct atomic formulas. We show that if $L$ is either $\mathsf{EC}$ or $\mathsf{ECN}$, the formula $\phi \to \psi$ is provable in $L$, while there is no formula $\theta$ such that $V(\theta) \subseteq \{q\}$ and both formulas $\phi \to \theta$ and $\theta \to \psi$ are provable in $L$.\\

\noindent To show that $\phi \to \psi$ is in $\mathsf{EC}$ and hence $\mathsf{ECN}$, we use the following proof tree in $\mathbf{GEC}$:
\begin{center}
\begin{tabular}{c c c}
 \AxiomC{$ p \wedge q, \neg q\wedge r \Rightarrow \bot$} 
  \AxiomC{$\bot \Rightarrow p \wedge q$}  
    \AxiomC{$\bot \Rightarrow \neg q\wedge r$} 
   \RightLabel{\footnotesize $EC$} 
 \TrinaryInfC{$\Box (p \wedge q), \Box (\neg q \wedge r) \Rightarrow \Box \bot$}
 \DisplayProof
\end{tabular}
\end{center}

\noindent Now, for the sake of contradiction, assume that the interpolant $\theta$ for $\phi \to \psi$ exists. Let $G$ be either 
$\mathbf{GEC}$ or $\mathbf{GECN}$. Hence, both 
$\Box (\neg q \wedge r) \Rightarrow \theta$ and $\Box (p \wedge q), \theta \Rightarrow \Box \bot$ 
are provable in $G$. We first analyse the general form of $\theta$.\\

\noindent First, note that by a simple induction on the structure of the formulas in the language $\mathcal{L}$, it is possible to show that any formula $A$ is $\mathbf{G3cp}$-equivalent to a CNF-style formula $\bigwedge_{i \in I} \bigvee_{j \in J_i} L_{ij}$, where $I$ and $J_i$'s are (possibly empty) finite sets, $V(L_{ij}) \subseteq V(A)$, and each $L_{ij}$ is either an atomic formula, the negation of an atomic formula, $\Box C$ or $\neg \Box C$, for a formula $C$. 
In particular, the formula $\theta$ is $\mathbf{G3cp}$-equivalent to a CNF-style formula in the form $\bigwedge_{i \in I} \bigvee_{j \in J_i} L_{ij}$. W.l.o.g, assume that for any $i \in I$, it is impossible to have both an atomic formula and its negation in $\{L_{ij}\}_{j \in J_i}$, and that none of sequents $(\, \Rightarrow L_{ij})$ or $( L_{ij} \Rightarrow \,)$ are provable in $G$. \\

Back to the main argument, as $\phi \Rightarrow \theta$ is provable in $G$, we have $\phi \Rightarrow \bigwedge_{i \in I} \bigvee_{j \in J_i} L_{ij} $ which means that for every $i \in I$, we have $\phi \Rightarrow \bigvee_{j \in J_i} L_{ij}$. Based on the form of each $L_{ij}$, we can transform the sequent to a provable sequent of the form $\phi, P, \Box \Gamma \Rightarrow Q, \Box \Delta$, where $P$ and $Q$ are multisets of atomic formulas and $\Gamma$ and $\Delta$ are multisets of formulas. We claim that for any $i \in I$, the corresponding $\Gamma$ is non-empty. Suppose $\Gamma=\varnothing$. Then, we have $\phi , P \Rightarrow Q, \Box \Delta$. This sequent must have been the conclusion of the rule $(EC)$, because for $G=\mathbf{GEC}$, the other possible case is being an axiom which implies either $\bot \in P$ or the existence of an atomic $s$ in $P \cap Q$. Both contradict 
the structure of $\bigvee_{j \in J_i} L_{ij}$. For $G=\mathbf{GECN}$, the same holds. Moreover, if the last rule is $(NW)$, then for an element 
$\delta \in \Delta$, the sequent $(\, \Rightarrow \delta)$ and hence $(\, \Rightarrow \Box \delta)$ must be provable in $G$ which contradicts 
the structure of $L_{ij}$'s again. Therefore, $T=(\phi , P \Rightarrow Q, \Box \Delta)$ is the consequence of $(EC)$ and hence, it has the form $(\Sigma, \Box \alpha_1, \cdots , \Box \alpha_n \Rightarrow \Box \beta, \Lambda)$ and the last rule is:
\begin{center}
\begin{tabular}{c c c}
 \AxiomC{$\alpha_1, \cdots , \alpha_n \Rightarrow \beta$} 
  \AxiomC{$\beta \Rightarrow \alpha_1 \;\;\;\;\; \cdots$}  
    \AxiomC{$\beta \Rightarrow \alpha_n$} 
   \RightLabel{\footnotesize $EC$} 
 \TrinaryInfC{$\Sigma, \Box \alpha_1, \cdots , \Box \alpha_n \Rightarrow \Box \beta, \Lambda$}
 \DisplayProof
\end{tabular}
\end{center}
Now there are two cases, 
either $\phi \in \Sigma$ or $\phi \notin \Sigma$. In the first case, as the formulas outside 
$\Sigma$ are 
either atomic or boxed, we must have no boxed formula outside 
$\Sigma$. 
This is impossible, as the form of the rule $(EC)$ dictates that we must have at least one boxed formula in the 
antecedent of the conclusion. Hence, $\phi \notin \Sigma$.\\
As all formulas 
in $T^{a}$ (except $\phi$) 
are atomic, we must have only one boxed formula in $T^a$, which is 
$\phi$. Therefore, in the premises of the rule, we have $\neg q \wedge r \Rightarrow \beta$ and $\beta \Rightarrow \neg q \wedge r$. Since $V(\beta) \subseteq V(\theta) \subseteq \{q\}$, 
then $\beta$ is $r$-free. If we once substitute $\bot$ for $r$ and 
then $\neg q$ for $r$, as $\beta$ remains intact, we will have $\beta \Leftrightarrow \bot$ and $\beta \Leftrightarrow \neg q$, which implies the contradictory $\bot \Leftrightarrow \neg q$. Hence, $\Gamma$ cannot be empty.\\

So far, we have proved that $\Gamma$ is non-empty, for any $i \in I$. Let $D_{i}$ be a formula in $\Gamma$ and note that $\neg \Box D_i$ occurs as one of the $L_{ij}$'s. Now, as $\Box (p \wedge q), \theta \Rightarrow \Box \bot$ or equivalently $\Box (p \wedge q), \bigwedge_{i \in I} \bigvee_{j \in J_i} L_{ij} \Rightarrow \Box \bot$ is provable in $G$, we have $\Box(p \wedge q), \{\neg \Box D_{i}\}_{i \in I} \Rightarrow \Box \bot$ is provable in $G$. Define $\mathcal{D}=\{D_{i}\}_{i \in I}$. Thus $S = (\Box (p \wedge q) \Rightarrow \Box \mathcal{D}, \Box \bot)$ is provable. As all the formulas are boxed, this must have been the conclusion of the rule $(EC)$. The reason is that $G$ has no weakening rules, and for $G=\mathbf{GEC}$, the only modal rule is ($EC$) and for $G=\mathbf{GECN}$, the last rule cannot be the rule $(NW)$ as it implies that for one $D \in \mathcal{D}$ the sequent $(\, \Rightarrow D)$ is provable in $G$ which means that $(\, \Rightarrow \Box D)$ and hence $(\neg \Box D \Rightarrow \,)$ is provable. The last contradicts with the structure of $L_{ij}$'s.
This implies that the last inference is of the form: 
\begin{center}
\begin{tabular}{c c c}
 \AxiomC{$\alpha_1, \cdots , \alpha_n \Rightarrow \beta$} 
  \AxiomC{$\beta \Rightarrow \alpha_1 \;\;\;\;\; \cdots$}  
    \AxiomC{$\beta \Rightarrow \alpha_n$} 
   \RightLabel{\footnotesize $EC$} 
 \TrinaryInfC{$\Sigma, \Box \alpha_1, \cdots , \Box \alpha_n \Rightarrow \Box \beta, \Lambda$}
 \DisplayProof
\end{tabular}
\end{center}
Similar as before, 
there are two cases, 
either $\beta=\bot$ or $\beta \in \mathcal{D}$.
If $\beta=\bot$, in the premises we must have $p \wedge q \Leftrightarrow \bot$ which is impossible. 
If $\beta \in \mathcal{D}$, it means that in the premises we 
had $p \wedge q \Leftrightarrow \beta$. Note that as $\beta \in \mathcal{D}$ we have $V(\beta) \subseteq V(\theta) \subseteq \{q\}$. Hence $\beta$ is $p$-free. Substituting once $\bot$ and then $q$ for $p$, leave $\beta$ intact and hence we get $\bot \Leftrightarrow \beta$ and $q \Leftrightarrow \beta$ which implies $q \Leftrightarrow \bot$, which is impossible. 
\end{proof}

For the conditional logics $\mathsf{CEC}$ and $\mathsf{CECN}$, let us first introduce the following two translations $t: \mathcal{L}_{\Box} \to \mathcal{L}_{\triangleright}$ and $s: \mathcal{L}_{\triangleright} \to \mathcal{L}_{\Box}$ that map atomic formulas and $\bot$ to themselves, commute with the propositional connectives and $(\Box \phi)^t=\top \triangleright \phi^t$ and $(\theta \triangleright \psi)^s=\Box \psi^s$, for any $\phi \in \mathcal{L}_{\Box}$ and $\theta , \psi \in \mathcal{L}_{\triangleright}$. Note that $(\phi^t)^s=\phi$, $V(\phi^t)=V(\phi)$, and $V(\psi^s) \subseteq V(\psi)$, for any $\phi \in \mathcal{L}_{\Box}$ and $\psi \in \mathcal{L}_{\triangleright}$. 

\begin{theorem}\label{Translations}
\begin{description}
\item[$(i)$]
If $\mathsf{EC} \vdash \phi$ then $\mathsf{CEC} \vdash \phi^t$.
\item[$(ii)$]
If $\mathsf{CEC} \vdash \psi$ then $\mathsf{EC} \vdash \psi^s$.
\end{description}
The same also holds for the pair $\mathsf{ECN}$ and $\mathsf{CECN}$. 
\end{theorem}
\begin{proof}
The proof is by a simple induction on the length of the proofs. As it is easy, we will only spell out one case. Suppose in $(i)$, we have $\mathsf{EC} \vdash \phi$ and we want to show that $\phi^t$ is provable in $\mathsf{CEC}$, where $\phi$ is an instance of the axiom $(C)$. By assumption, $\phi$ is of the form $\Box A \wedge \Box B \to \Box (A \wedge B)$, where the formulas $A$ and $B$ are in the language $\mathcal{L}_{\Box}$. By definition, as $t$ commutes with the propositional connectives, $\phi^t$ will be of the form $(\top \triangleright A^t) \wedge (\top \triangleright B^t) \to (\top \triangleright A^t \wedge B^t)$. The latter is an instance of the axiom $(CC)$, as the formulas $A^t$ and $B^t$ are in the language $\mathcal{L}_{\triangleright}$. Therefore, $\phi^t$ is provable in $\mathsf{CEC}$.
\end{proof}

\begin{theorem}
Logics $\mathsf{CEC}$ and $\mathsf{CECN}$ do not enjoy Craig interpolation.
As a consequence, they do not have uniform or uniform Lyndon interpolation property.
\end{theorem}
\begin{proof}
We prove the claim for $\mathsf{CEC}$. The other is similar. Assume $\mathsf{CEC}$ has CIP.  We prove that $\mathsf{EC}$ also enjoys CIP which contradicts Theorem \ref{Thm:CIP EC ECN}. 
Set $\phi=\Box (\neg q \wedge r)$ and $\psi= \Box (p \wedge q) \to \Box \bot$, as in the proof of Theorem \ref{Thm:CIP EC ECN}, where it was shown that $\mathsf{EC} \vdash \phi \to \psi$, while there is no formula $\theta$ such that $V(\theta) \subseteq \{q\}$ and both formulas $\phi \to \theta$ and $\theta \to \psi$ are provable in $\mathsf{EC}$.
Since $\mathsf{EC} \vdash \phi \to \psi$, by the part $(i)$ of Theorem \ref{Translations}, we have $\mathsf{CEC} \vdash \phi^t \to \psi^t$. Hence, using the assumption that $\mathsf{CEC}$ has CIP, there is a formula $\chi \in \mathcal{L}_{\triangleright}$ such that $V(\chi) \subseteq V(\phi^t) \cap V(\psi^t)=V(\phi) \cap V(\psi)$ and $\mathsf{CEC} \vdash \phi^t \to \chi$ and $\mathsf{CEC} \vdash \chi \to \psi^t$. Now, apply the translation $s$. By the part $(ii)$ of Theorem \ref{Translations}, we have $\mathsf{EC} \vdash (\phi^t)^s \to \chi^s$ and $\mathsf{EC} \vdash \chi^s \to (\psi^t)^s$. As $(\phi^t)^s=\phi$, $(\psi^t)^s=\psi$, and $V(\chi^s) \subseteq V(\chi)$, the formula $\chi^s$ would be the interpolant for $\phi \to \psi$ in $\mathsf{EC}$, which is a contradiction.
\end{proof}
Now, to prove the last contribution of this section, namely that the logics $\mathsf{CKCEM}$ and $\mathsf{CKCEMID}$ do not enjoy uniform Lyndon interpolation property, first, let us make the following observation on the peculiarity of the combination of uniform Lyndon interpolation property and the axiom $(CEM)$:
\begin{theorem}\label{NegConditional}
Let $L \supseteq \mathsf{CKCEM}$ be a logic. Then, if $L$ enjoys uniform Lyndon interpolation property, then $L \vdash (p \triangleright r) \vee (q \triangleright \neg r)$.
\end{theorem}
\begin{proof}
Let $p$, $q$ and $r$ be three distinct atoms. Using the axiom $(CEM)$, we have $\mathsf{CKCEM} \vdash (q \triangleright r) \vee (q \triangleright \neg r)$, and therefore
$L \vdash (q \triangleright r) \vee (q \triangleright \neg r)$. Then, $L \vdash \neg (q \triangleright r) \to (q \triangleright \neg r)$. Note that $V^+(q \triangleright \neg r)=\varnothing$ and $V^-(\neg (q \triangleright r))=\{r\}$. Set $\phi=\forall^{-} q \, (q \triangleright \neg r)$. By the definition of uniform Lyndon interpolation, $\phi$ is $q^-$-free, $L \vdash \phi \to (q \triangleright \neg r)$ and $V^{+}(\phi) \subseteq V^+(q \triangleright \neg r)=\varnothing$. By the latter, $\phi$ is $q^+$-free. Hence, $\phi$ is $q$-free which means that $q$ does not occur in $\phi$. Again, by the definition of uniform Lyndon interpolation and the facts that $\neg (q \triangleright r)$ is $q^-$-free and $L \vdash \neg (q \triangleright r) \to (q \triangleright \neg r)$, we have $L \vdash \neg (q \triangleright r) \to \phi$, using the property $(ii)$ in Definition \ref{DfnUniformInterpolation}. As $\phi$ is $q$-free, we can substitute $q$ by $p$, keeping $\phi$ intact. Hence, $L \vdash \neg (p \triangleright r) \to \phi$. Finally, as $L \vdash \phi \to (q \triangleright \neg r)$, we have $L \vdash \neg (p \to r) \to (q \triangleright \neg r)$. Hence, $L \vdash (p \triangleright r) \vee (q \triangleright \neg r)$.
\end{proof}

\begin{corollary}
Logics $\mathsf{CKCEM}$ and $\mathsf{CKCEMID}$ do not enjoy uniform Lyndon interpolation property.
\end{corollary}
\begin{proof}
Let $L=\mathsf{CKCEM}$ (resp., $\mathsf{CKCEMID}$) and $G$ be the corresponding sequent system for $L$ provided in Section \ref{Preliminaries}. As $L \supseteq \mathsf{CKCEM}$, by Theorem \ref{NegConditional}, if $L$ enjoys uniform Lyndon interpolation property, then $L \vdash (p \triangleright r) \vee (q \triangleright \neg r)$ which implies the provability of the sequent $(\, \Rightarrow p \triangleright r, q \triangleright \neg r)$ in $G$. Call this proof $\cald$. As the sequent is not an axiom, the last rule in $\cald$ is either the weakening rule or $(CKCEM)$ (resp., $(CKCEMID)$). In the latter case, the proof looks like the following
\begin{center}
\begin{tabular}{c c}
\AxiomC{$\cald'$}
\noLine
 \UnaryInfC{$p \Leftrightarrow q$} 
 \AxiomC{$\cald''$}
 \noLine
    \UnaryInfC{$\Rightarrow r, \neg r$}
   \RightLabel{\footnotesize$CKCEM$} 
 \BinaryInfC{$\Rightarrow p \triangleright r , q \triangleright \neg r$}
 \DisplayProof
 &
 \AxiomC{$\cald'$}
\noLine
 \UnaryInfC{$p \Leftrightarrow q$} 
 \AxiomC{$\cald''$}
 \noLine
    \UnaryInfC{$p,q \Rightarrow r, \neg r$}
   \RightLabel{\footnotesize$CKCEMID$} 
 \BinaryInfC{$\Rightarrow p \triangleright r , q \triangleright \neg r$}
 \DisplayProof
\end{tabular}
\end{center}
\noindent By a simple combinatorial case check on the possible cases of the last rule in the proof $\cald'$, we can easily show that such a proof cannot exist. Therefore, the last rule in the proof $\cald$ cannot be $(CKCEM)$ (resp., $(CKCEMID)$). In the case that the last rule in $\cald$ is the weakening rule, either $\Rightarrow (p \triangleright r)$ or $\Rightarrow(q \triangleright \neg r)$ is provable in $G$.
As none of these sequents is an axiom and the last rule of their proofs cannot be the weakening rule, then the last rule must be $(CKCEM)$ (resp., $(CKCEMID)$). Therefore, since in the premises of the rule $(CKCEM)$ (resp., $(CKCEMID)$), both $I$ and $J$ are empty, then we either have $G \vdash (\, \Rightarrow r)$ or $G \vdash (\, \Rightarrow \neg r)$ (resp., $G \vdash (p \Rightarrow r)$ or $G \vdash (q \Rightarrow \neg r)$). Again, it is easy to check that none of theses cases can take place.
\end{proof}

\bibliographystyle{plainurl}
\bibliography{newbib}


\end{document}